\RequirePackage{fix-cm}

\documentclass[smallextended,envcountsame,envcountsect]{svjour3} 

\smartqed  
\usepackage{graphicx}
 \usepackage{mathptmx}      

\usepackage{amsmath,amsfonts,amsbsy,amsgen,amscd,mathrsfs,amssymb,mathtools}

\usepackage[colorlinks=true,citecolor=blue,linkcolor=blue]{hyperref}
\usepackage[top=30mm, bottom=30mm, left=45mm, right=45mm]{geometry}
\usepackage{relsize}

\newcommand\al{\alpha}
\newcommand\be{\beta}
\newcommand\dd{\mathrm d}
\newcommand\De{\Delta}
\newcommand\de{\delta}
\newcommand\deq{\stackrel{\mathrm{distr.}}{=}}
\newcommand\eps{\varepsilon}

\newcommand\ka{\kappa}
\newcommand\La{\Lambda}
\newcommand\la{\lambda}

\newcommand\Om{\Omega}

\newcommand\si{\sigma}

\newcommand\ze{\zeta}
\renewcommand\d{~\mathrm d}
\renewcommand\phi{\varphi}
\renewcommand\th{\vartheta}
\renewcommand\hat{\widehat}

\newcommand\mbb{\mathbb}
\newcommand\mbf{\mathbf}
\newcommand\mc{\mathcal}
\newcommand\mf{\mathfrak}
\newcommand\mr{\mathrm}
\newcommand\ms{\mathscr}

\numberwithin{equation}{section}

\makeatletter\renewenvironment{proof}[1][\proofname]
{\par
\normalfont\topsep6\p@\@plus6\p@\relax\trivlist\item[\hskip\labelsep\bfseries#1\@addpunct{.}]\ignorespaces}\qed

\journalname{}
\begin{document}

\title{Phase Transitions in Asymptotically Singular Anderson Hamiltonian and Parabolic Model
}

\titlerunning{Phase Transitions in Asymptotically Singular AH and PAM}        

\author{Pierre Yves Gaudreau Lamarre}


\institute{P. Y. Gaudreau Lamarre\at
University of Chicago, Chicago, IL\space\space60637, USA\\
              \email{pyjgl@uchicago.edu}
}

\date{}

\maketitle

\begin{abstract}
Let $\xi$ be a Gaussian white noise on $\mbb R^d$ ($d=1,2,3$).
Let $(\xi_\eps)_{\eps>0}$ be continuous Gaussian
processes such that $\xi_\eps\to\xi$ as $\eps\to0$, defined by convolving
$\xi$ against a mollifier. We consider the asymptotics
of the parabolic Anderson model (PAM) with noise $\xi_{\eps(t)}$ for large time $t\gg1$,
and the Dirichlet eigenvalues of the Anderson Hamiltonian (AH) with potential $\xi_{\eps(t)}$
on large boxes $(-t,t)^d$, where the parameter
$\eps(t)$ vanishes as $t\to\infty$. We prove that the asymptotics in question
exhibit a phase transition in the rate at which $\eps(t)$ vanishes,
which distinguishes between the behavior observed in the AH/PAM with continuous
Gaussian noise and white noise.
By comparing our main theorems with previous results on the AH/PAM
with white noise, our results show that some asymptotics of
the latter can be accessed with solely elementary
methods, and we obtain quantitative estimates on the difference between
the AH/PAM with white noise and its continuous-noise approximations
as $t\to\infty$.
\end{abstract}

\section{Introduction}

\subsection{Continuous PAM and AH}

The continuous parabolic Anderson model (PAM) is defined as the
solution $u(t,x)$ of a random heat equation of the form
\begin{align}
\label{Equation: General PAM}
\begin{cases}
\partial_tu(t,x)=\tfrac12\De u(t,x)+\xi(x)u(t,x)\\
u(0,x)=u_0(x)
\end{cases},
\qquad t\geq0\text{ and }x\in\mbb R^d,
\end{align}
where $\xi$ is a random potential called the noise.
A closely associated object is the Anderson Hamiltonian\footnote{In terms of physical terminology,
one should instead define the Anderson Hamiltonian with random potential $\xi$ as $-\tfrac12\De+\xi$.
That said, in this paper we use \eqref{Equation: General AH}
for convenience.}
(AH), defined as the operator
\begin{align}
\label{Equation: General AH}
Af(x):=\tfrac12\De f(x)+\xi(x)f(x)
\end{align}
acting on some domain of functions $f:\mbb R^d\to\mbb R$
on which $A$ is self-adjoint.

Starting from the pioneering work
of G\"artner and Molchanov \cite{GartnerMolchanov},
the PAM literature has mostly been concerned with understanding the
occurrence of intermittency in \eqref{Equation: General PAM} for
large times (e.g., \cite[Section 1.4]{KonigBook}). Given the connection
between the AH and PAM via semigroup theory, a closely related
problem is that of localization in the AH's spectrum
(e.g., \cite[Sections 2.2.1--2.2.4]{KonigBook}).
We refer to \cite{CarmonaMolchanovBook,KonigBook} and references
therein for surveys of the field.
As it turns out, a few features of the AH/PAM have been the subject
of the majority of investigations to date, arguably due to the fact that they are
amenable to computation and encode useful information about
the geometry of intermittency:
Let us denote the Dirichlet eigenvalues of $A$ on a bounded open set $\Om\subset\mbb R^d$ as
\begin{align}
\label{Equation: Eigenvalues}
\La_1(A,\Om)\geq\La_2(A,\Om)\geq\La_3(A,\Om)\geq\cdots.
\end{align}
Let us define the total mass of the PAM as
\begin{align}
\label{Equation: Total Mass}
U(t):=\mbf E^0\left[\exp\left(\int_0^t\xi\big(B(s)\big)\d s\right)\right],\qquad t\geq0,
\end{align}
where $B$ is a standard Brownian motion on $\mbb R^d$
and $\mbf E^0$ denotes the expectation with respect to $B$ with initial
value of zero (i.e., $B(0)=0$), conditional on $\xi$.
(Equivalently, we can write the total mass $U(t)=u(t,0)$ as the solution of the PAM at $x=0$
with flat initial condition $u(0,x)=1$.)

\begin{problem}[Annealed Total Mass]
\label{Problem: Annealed}
Understand the $t\to\infty$ behavior of
the moments of the total mass $\mbf E\big[U(t)^p\big]$ ($p\geq0$).
\end{problem}
\begin{problem}[Quenched Total Mass]
\label{Problem: Quenched}
Understand the almost-sure $t\to\infty$ behavior of
the total mass $U(t)$.
\end{problem}
\begin{problem}[Eigenvalues]
\label{Problem: Eigenvalue}
Understand the almost-sure $t\to\infty$ behavior of the eigenvalues
$\La_k(A,Q_t)$ for fixed $k\geq1$ on large boxes $Q_t:=(-t,t)^d$.
\end{problem}
We refer to \cite{CarmonaMolchanov,GartnerKonig,GartnerKonigMolchanov}
for a derivation of the first- and second-order asymptotics of the above
in the AH/PAM with certain continuous noises (including
continuous Gaussian processes)
and an explanation of how these computations shed light on the
geometry of intermittency. See \cite{GartnerMolchanov,GartnerMolchanov2}
for similar results in the discrete setting.

\subsection{AH/PAM with White Noise}

In this paper, we are interested in understanding intermittency
in the PAM with white noise (WN).
WN is formally defined as a centered Gaussian process on $\mbb R^d$
with delta Dirac covariance
\begin{align}
\label{Equation: Intuitive White Noise}
\mbf E\big[\xi(x)\xi(y)\big]\,``="\,\de_0(x-y),\qquad x,y\in\mbb R^d.
\end{align}
Although WN is among the most natural examples
of noises to consider on $\mbb R^d$ (e.g., \cite[Section 1.5.2]{KonigBook}),
the rigorous treatment of the AH/PAM in this setting is made difficult
by the fact that WN is a Schwartz distribution.
Most notably, the $``$pointwise products$"$ $\xi(x)u(t,x)$ and $\xi(x)f(x)$ in
\eqref{Equation: General PAM} and \eqref{Equation: General AH} are ill posed,
making the very definition of the AH/PAM nontrivial.

To overcome this technical issue, an intuitive approach
is to proceed as follows:
Using classical theory,
define a family of approximate AHs and PAMs
$(A_\eps)_{\eps>0}$ and $(u_{\eps})_{\eps>0}$ with smoothed noises
$(\xi_\eps)_{\eps>0}$ that approach $\xi$ as $\eps\to0$.
Then, the hope is that we can obtain universal
(i.e., independent of the particular way in which we define $\xi_\eps$) limits
\begin{align}
\label{Equation: Naive Limit}
A:=\lim_{\eps\to0} A_\eps
\qquad\text{and}\qquad
u(t,x):=\lim_{\eps\to0}u_\eps(t,x),
\end{align}
which we take as the definitions of the AH/PAM with WN.
In one dimension ($d=1$), this procedure works
and there is a straightforward sense in which the limits \eqref{Equation: Naive Limit}
can be interpreted as the AH/PAM using quadratic forms and stochastic calculus; see \cite{BloemendalVirag,Chen14,FukushimaNakao,GaudreauLamarre2,GaudreauLamarre,GaudreauLamarreShkolnikov,Tindeletal,Labbe,RamirezRiderVirag}. 
In contrast, in higher dimensions ($d\geq2$) the limits \eqref{Equation: Naive Limit}
blow up. While the AH/PAM with WN are not expected to
make sense for $d\geq4$ (e.g., \cite{Hairer,Labbe}),
for $d=2,3$ nontrivial limits can be obtained
if one considers renormalizations of
 $u_\eps$ and $A_\eps$.
The limits thus obtained can be
interpreted in a rigorous sense as
the AH/PAM with WN using sophisticated solution theories for
SPDEs with irregular noise, such as regularity structures or
paracontrolled calculus; e.g., \cite{AllezChouk,GIP,Hairer,HairerLabbe2,Labbe}.
(Though, in some cases, simpler constructions can be used,
e.g., \cite{HairerLabbe}.)

Due to these technical difficulties,
the understanding of intermittency in the PAM with WN
is much less advanced than that with
continuous Gaussian noise
(c.f., \cite{CarmonaMolchanov,GartnerKonig,GartnerKonigMolchanov}).
More specifically, for $d=1$, first-order asymptotics
for Problems \ref{Problem: Annealed}--\ref{Problem: Eigenvalue} have been
obtained in \cite{Chen14,Tindeletal} (see also
\cite{CambroneroMcKean,CambroneroRiderRamirez,DumazLabbe,McKean}).
For $d=2,3$, it is understood that the total mass moments blow up in finite
time \cite{AllezChouk,ChenTindelMoments,Labbe} (and thus Problem
\ref{Problem: Annealed} is intractable), first-order asymptotics for
Problems \ref{Problem: Quenched} and \ref{Problem: Eigenvalue} when $d=2$
were proved in \cite{ChoukVZ,KPZ} using paracontrolled calculus, 
and Problems \ref{Problem: Quenched} and \ref{Problem: Eigenvalue} for $d=3$ are open.

\subsection{Main Results}

We now proceed to an exposition of our main results
(Theorems \ref{Theorem: Main Regular},
\ref{Theorem: Main Singular}, and \ref{Theorem: Annealed}
below). For the remainder of this paper, unless otherwise mentioned,
we assume that $d\in\{1,2,3\}$.

\subsubsection{Asymptotically Singular Noise}

Throughout the paper, we consider the following type of smoothed noise:

\begin{definition}
$\xi_1$ is a continuous, centered,
and stationary Gaussian process on $\mbb R^d$ with covariance
\begin{align}
\label{Equation: xi1 Covariance}
\mbf E\left[\xi_1(x)\xi_1(y)\right]=R(x-y),\qquad x,y\in\mbb R^d.
\end{align}
We assume that we can write
$R=\bar R*\bar R$,
where the function $\bar R:\mbb R^d\to\mbb R$ satisfies the following conditions:
\begin{enumerate}
\item $\bar R$ is a probability density function,
\item $\bar R$ is an even function,
\item $\bar R$ is compactly supported, and
\item there exists some $h>0$ and $C>0$ such that
\begin{align}
\label{Equation: Holder Constant}
|\bar R(x)-\bar R(y)|\leq C|x-y|_2^h\qquad
\text{for every }x,y\in\mbb R^d.
\end{align}
\end{enumerate}
\end{definition}

Then, for every $\eps\in(0,1]$, we define the approximate AH and
PAM total mass as
\begin{align}
\label{Equation: Approx AH and PAM}
A_\eps:=\tfrac12\De+\xi_\eps
\qquad\text{and}\qquad
U_\eps(t):=\mbf E^0\left[\exp\left(\int_0^t\xi_\eps\big(B(s)\big)\d s\right)\right],
\end{align}
where $\xi_\eps(x):=\eps^{-d/2}\xi_1(x/\eps)$. We denote the Dirichlet eigenvalues
of $A_\eps$ on some bounded open set $\Om\subset\mbb R^d$ as
$\La_1(A_\eps,\Om)\geq\La_2(A_\eps,\Om)\geq\cdots$.

\begin{remark}
If we denote
\[\bar R_\eps(x):=\eps^{-d}\bar R(x/\eps)
\qquad\text{and}\qquad
R_\eps(x):=\eps^{-d}R(x/\eps),\]
then $R_\eps=\bar R_\eps*\bar R_\eps$, and $\xi_\eps$ has
covariance $R_\eps$. Hence, $\xi_\eps\deq\xi*\bar R_\eps$,
where $\xi$ is a WN. Though it is more common to
use $\xi*\bar R_\eps$ as the definition of the smoothed noise, in this paper we use
the coupling $\xi_\eps(x)=\eps^{-d/2}\xi_1(x/\eps)$ for convenience,
as doing so does not affect our main results (see Remark \ref{Remark: In Probability}).
\end{remark}

Our aim in this paper is to propose to study the large-$t$ asymptotics
of the AH/PAM with WN by considering {\it asymptotically singular noise}.
That is, we study the behavior of $\La_k(A_{\eps(t)},Q_t)$ and $U_{\eps(t)}(t)$
as $t\to\infty$, where the approximation parameter $\eps(t)$ goes to zero as
$t\to\infty$. The hope is that
\begin{enumerate}
\item if $\eps(t)\to0$ at a fast enough rate, then the
asymptotics of $\La_k(A_{\eps(t)},Q_t)$ and $U_{\eps(t)}(t)$
carry insight into those of the AH/PAM with WN, and
\item since we are only ever considering objects with
continuous noise, the asymptotics of $\La_k(A_{\eps(t)},Q_t)$
and $U_{\eps(t)}(t)$ can be accessed with
elementary methods (at least comparatively
to regularity structures/paracontrolled calculus).
\end{enumerate}

\subsubsection{Quenched Phase Transitions}

In this paper, we take the first steps in actualizing the above-described
program: Using only elementary methods
(i.e., standard operator/semigroup theory, suprema of continuous Gaussian processes, etc.),
we prove that the first-order asymptotics in Problems \ref{Problem: Quenched} and
\ref{Problem: Eigenvalue}
exhibit a $``$phase transition$"$ in the rate $\eps(t)$ at which $\xi_{\eps(t)}$ becomes singular
as $t\to\infty$. To this effect, our first main result states that if $\eps(t)$ is not too small,
then the first-order quenched total mass and eigenvalue
asymptotics behave as though $\eps(t)$ is constant.
We call this regime of $\eps(t)$ the regular phase.

\begin{definition}[Regular Phase]
The function $\eps(t)\in(0,1]$ ($t\geq0$) is in the regular phase
if $\eps(t)\gg(\log t)^{-1/(4-d)}$ as $t\to\infty$.
\end{definition}

\begin{theorem}[Regular Phase]
\label{Theorem: Main Regular}
Let $\eps(t)$ be in the regular phase.
\begin{align}
\label{Equation: Main Regular Eigenvalue}
\lim_{t\to\infty}\frac{\La_k(A_{\eps(t)},Q_t)}{\eps(t)^{-d/2}\sqrt{\log t}}
=\sqrt{2dR(0)}\qquad\text{in probability}
\end{align}
for every $k\in\mbb N$, and
\begin{align}
\label{Equation: Main Regular TM}
\lim_{t\to\infty}\frac{\log U_{\eps(t)}(t)}{t\,\eps(t)^{-d/2}\sqrt{\log t}}
=\sqrt{2dR(0)}\qquad\text{in probability}.
\end{align}
\end{theorem}

\begin{remark}
If we take $\eps(t)=1$ in Theorem \ref{Theorem: Main Regular},
then we recover the first-order asymptotics for the PAM with continuous
Gaussian noise $\xi_1$ in \cite[Theorem 5.1]{CarmonaMolchanov}.
\end{remark}

Our second main result states that if $\eps(t)\to0$ at a fast enough rate,
then the quenched total mass and eigenvalue
asymptotics are universal (i.e., independent of the
choice of $R$), and are given by a variational constant.
Moreover, this result identifies
$(\log t)^{-1/(4-d)}$ as the critical rate of decay at which
this transition occurs. We call this second regime the
singular phase.

\begin{definition}[Singular Phase]
$\eps(t)\in(0,1]$ is in the singular phase if
one of the following holds:
\begin{enumerate}
\item $d=1$ and $\eps(t)\ll(\log t)^{-1/(4-d)}$ as $t\to\infty$; or
\item $d=2,3$, the H\"older exponent $h$ in \eqref{Equation: Holder Constant}
satisfies $h>d/4$, and
\[(\log t)^{-1/(4-d)-\mf c_d}\ll\eps(t)\ll(\log t)^{-1/(4-d)}\]
as $t\to\infty$, where we define the constant
\begin{align}
\label{Equation: Singular Phase Lower Constant}
\mf c_d:=\frac{h}{d(d+h)}.
\end{align}
\end{enumerate}
\end{definition}

\begin{definition}[Variational Constant]
Let $\mf G_d\in(0,\infty)$ be the smallest possible constant in the
Gagliardo-Nirenberg-Sobolev
(GNS) inequality
\begin{align}
\label{Equation: GNS Intro}
\|\phi\|_4^4\leq \mf G_d\left(\int_{\mbb R^d}|\nabla\phi(x)|_2^2\d x\right)^{d/2}\|\phi\|_2^{4-d}
\qquad\text{for all }\phi\in C_0^\infty(\mbb R^d)
\end{align}
(since $2(d-2)\leq d$ holds for $d=1,2,3$,
we know that $\mf G_d<\infty$; e.g., \cite[(C.1)]{ChenBook}).
Then, we define the associated Lyapunov exponent
\begin{align}
\label{Equation: Lyapunov}
\mf L_d:=\frac{4-d}4\left(\frac d2\right)^{d/(4-d)}(2d\mf G_d)^{2/(4-d)}.
\end{align}
\end{definition}

\begin{theorem}[Singular Phase]
\label{Theorem: Main Singular}
Let $\eps(t)$ be in the singular phase.
\begin{align}
\label{Equation: Main Singular Eigenvalue}
\lim_{t\to\infty}\frac{\La_k(A_{\eps(t)},Q_t)}{(\log t)^{2/(4-d)}}
=\mf L_d\qquad\text{in probability}
\end{align}
for every $k\in\mbb N$, and
\begin{align}
\label{Equation: Main Singular TM}
\lim_{t\to\infty}\frac{\log U_{\eps(t)}(t)}{t\,(\log t)^{2/(4-d)}}
=\mf L_d\qquad\text{in probability}.
\end{align}
\end{theorem}

We now end the statement of Theorems \ref{Theorem: Main Regular}
and \ref{Theorem: Main Singular} with some remarks:

\begin{remark}
When $d=1,2$, the asymptotics in Theorem
\ref{Theorem: Main Singular} match that of the
AH/PAM with WN proved by Chen, Chouk, K\"onig, Perkowski, and van Zuijlen in \cite{Chen12,ChoukVZ,KPZ}.
We refer to Section \ref{Section: Applications} for more
details on the applications of our results to the
AH/PAM with WN.
\end{remark}

\begin{remark}
When $d\geq4$, we can prove that no phase transition occurs.
More specifically, the asymptotics remain in the regular phase no how matter quickly
$\eps(t)\to0$ (see Remark \ref{Remark: Phase Transition Explanation}
for a heuristic and Theorem \ref{Theorem: d4} and
Sections \ref{Section: Remark on d4 1} and \ref{Section: Remark on d4 2}
for specifics). In particular, this lack of phase transition provides a different
point of view with which to explain that the AH/PAM with WN do not
make sense in $d\geq4$.
\end{remark}

\begin{remark}
\label{Remark: Renormalization}
The lower bound of $(\log t)^{-1/(4-d)-\mf c_d}\ll\eps(t)$ in the singular phase for $d=2,3$
is due to the fact that a technical argument fails when $\eps(t)$ is too small
(see \eqref{Equation: Continuous Dudley}
in Proposition \ref{Proposition: Continuous Upper Bound Asymptotics}).
While we make no claim that \eqref{Equation: Singular Phase Lower Constant} is optimal for
Theorem \ref{Theorem: Main Singular} to hold, some kind of lower bound
is to be expected in $d=2,3$, since in those cases the AH/PAM must be renormalized
to obtain nontrivial $\eps\to0$ limits.
We point to Section \ref{Section: Applications} below for more details on this point.
\end{remark}

\begin{remark}
It would be interesting to see if asymptotics that interpolate between
Theorems \ref{Theorem: Main Regular} and \ref{Theorem: Main Singular}
could be obtained in a ``critical phase" of the form
\[\eps(t)=C(\log t)^{-1/(4-d)}\big(1+o(1)\big),\qquad t\to\infty\]
for some $C>0$. As we were unable to obtain matching upper and lower bounds in this
regime, we leave it as an open question.
\end{remark}

\begin{remark}
\label{Remark: In Probability}
When comparing Theorems \ref{Theorem: Main Regular}
and \ref{Theorem: Main Singular} with the corresponding
results in \cite{CarmonaMolchanov,Chen12,ChoukVZ,KPZ},
it can be noted that the former all prove almost sure convergence,
whereas in this paper we only prove convergence in probability.
The argument typically used to prove almost sure convergence
in this context relies on the monotonicity of $\La_k(H,Q_t)$ in $t$
for every fixed operator $H$. Given that, in this paper, the operator
$A_{\eps(t)}$ changes with $t$, this argument can no longer be used
(except in the special case of the subcritical phase where $\eps(t)$ is constant).
\end{remark}

\subsubsection{Annealed Total Mass}

As mentioned earlier in this introduction, the moments of the PAM
total mass with WN are not finite for all $t>0$ in $d=2,3$. That
said, it is nevertheless natural to ask if the moments
$\mbf E\big[U_{\eps(t)}(t)^p\big]$ carry meaningful information
about intermittency in the PAM with WN when $\eps(t)$
is very small. The following result suggests that this may not the case
when $d=2,3$:

\begin{theorem}
\label{Theorem: Annealed}
Let $\eps(t)\in(0,1]$ for $t\geq0$.
On the one hand, if
\begin{enumerate}
\item $d=1$ and $\eps(t)\gg t^{-1}$, or
\item $d\geq2$,
\end{enumerate}
then for every $p\in\mbb N$,
\begin{align}
\label{Equation: Regular Annealed Total Mass}
\lim_{t\to\infty}\frac{\log\mbf E\big[U_{\eps(t)}(t)^p\big]}{\eps(t)^{-d}t^2}=\frac{p^2R(0)}{2}.
\end{align}
On the other hand, if $d=1$ and $\eps(t)\ll t^{-1}$ as $t\to\infty$,
then there exists some constants $0<\theta_1\leq\theta_2<\infty$
independent of $R$ such that
for every $p\in\mbb N$,
\begin{align}
\label{Equation: Singular Annealed Total Mass}
\theta_1p^3\leq\liminf_{t\to\infty}\frac{\log\mbf E\big[U_{\eps(t)}(t)^p\big]}{t^3}\leq
\limsup_{t\to\infty}\frac{\log\mbf E\big[U_{\eps(t)}(t)^p\big]}{t^3}\leq\theta_2p^3.
\end{align}
\end{theorem}

\begin{remark}
When $\eps(t)=1$, Theorem \ref{Theorem: Annealed} reduces to
the moment asymptotics \cite[Theorem 4.1]{CarmonaMolchanov}
for the PAM with continuous Gaussian noise $\xi_1$.
In particular, when $d\geq2$, Theorem \ref{Theorem: Annealed}
shows that no matter how small we take $\eps(t)$,
the moment asymptotics never transition to a universal limit
independent of $R$.
In contrast, if $d=1$ and $\eps(t)\ll t^{-1}$, then we recover
in \eqref{Equation: Singular Annealed Total Mass}
the annealed asymptotics for the one-diemsional PAM with WN proved in \cite[(6.8)]{Tindeletal}.
\end{remark}

\subsection{Other Results and Applications}

We now discuss how our paper relates to the wider
literature, taking this opportunity to showcase an application of our results to
the study of the AH/PAM with WN in Corollary \ref{Corollary: Application}.

\subsubsection{Other Noise Scalings}

We note that the present paper is not the first
to study spectral asymptotics of AH- and PAM-type
objects whose noise depends on the parameter being sent to infinity
(or zero),
including the occurrence of a phase transition.

One the one hand, in \cite{MerklWuthrich1,MerklWuthrich2} Merkl and W\"uthrich
consider the largest eigenvalue of the operator $\frac12\De-V_t$
on the box $Q_t$, where $V_t$ is of the form
\[V_t(x):=\frac{\be}{\phi(t)^{2}}\sum_iW(x-x_i),\qquad x\in\mbb R^d\]
for some scale function $\phi$, shape function $W$, $\be>0$,
and Poisson point process $(x_i)_i$.
In particular, they identify the presence of a phase
transition in the asymptotics when $d\geq4$ and $\phi(t)=(\log t)^{1/d}$
in terms of the parameter $\be$ (i.e., there exists a dimension-dependent critical
$\be_c>0$ such that the asymptotics differ if $\be<\be_c$ or $\be>\be_c$).
Then, they study the
behavior of the so-called Brownian motion in the scaled
Poissonian potential $V_t$ (the analog of the PAM in their setting).
Although the broad outline of the strategies used in those papers
(especially \cite{MerklWuthrich1}) is similar to the present paper,
the details are very different due to the nature of the random potentials
and their scaling in $t$
(i.e., a Poissonian potential with a multiplicative factor
versus a Gaussian potential with both a multiplicative factor
and a space scaling).

On the other hand, in \cite{BFG1,BFG2}, Biskup, Fukushima, and K\"{o}nig study the $\eps\to0$ asymptotics
of the top eigenvalues of discrete operators of the form $\eps^{-2}\De-\xi^{(\eps)}$
on lattices that approximate some bounded domain $D\subset\mbb R^d$ (i.e.,
the space between lattice points is of order $\eps$). Here, $\De$ is the lattice
Laplacian, and $\xi^{(\eps)}$ is an independent random field with a properly scaled expectation
and variance. More specifically, they establish convergence in probability of the eigenvalues
of $\eps^{-2}\De-\xi^{(\eps)}$ to that of a deterministic $``$homogenized$"$ continuum Schr\"odinger
operator on $D$ whose potential is determined by $\mbf E[\xi^{(\eps)}]$. They also prove
Gaussian fluctuations of the eigenvalues about their expectations, where the limiting
covariance depends on $\mbf{Var}[\xi^{(\eps)}]$. While the setting in \cite{BFG1,BFG2} differs
from the present paper in various significant ways (in particular, both the domain of the operators
and the variance $\mbf{Var}[\xi^{(\eps)}]$ are bounded as $\eps\to0$), it nevertheless raises the interesting question
of whether analogs of Theorems \ref{Theorem: Main Regular} and
\ref{Theorem: Main Singular} can be proved if the approximations $A_{\eps(t)}$ and $U_{\eps(t)}$
are constructed using properly scaled lattice WN instead of a continuous smoothing
of the WN. If that is possible, then some ideas from \cite{BFG1,BFG2} would be expected
to be fruitful; we leave this direction open.

\subsubsection{Finer Asymptotics}

As shown in \eqref{Equation: Main Regular Eigenvalue}
and \eqref{Equation: Main Singular Eigenvalue},
the eigenvalues $\La_k(A_{\eps(t)},Q_t)$ are all asymptotically
equivalent in the first order (as $t\to\infty$). We expect that,
with a finer scaling, one could identify
the fluctuations of the eigenvalues and thus uncover a
nontrivial point process limit. Such a result would complement
previous investigations in this direction, such as the Poisson
point process limits uncovered by Dumaz and Labb\'e in \cite{DumazLabbe} for the
one-dimensional WN (see also \cite{Astrauskas} and references
therein for a survey of such results in the discrete setting).

In a similar vein, it is natural to wonder if a phase transition
also occurs in the second order asymptotics of the PAM.
If that is the case, then we expect that, in similar fashion to
\cite{GartnerKonig,GartnerKonigMolchanov}, understanding the
transition of the smaller order asymptotics would provide new information
on intermittency in the PAM with WN;
more specifically, in clarifying the connection (if any) between the
variational constant $\mf L_d$ in \eqref{Equation: Lyapunov} and the
local geometry of intermittent peaks.

We leave both of these questions open for future investigations.

\subsubsection{Renormalizations and Rate of Convergence}
\label{Section: Applications}

As mentioned in Remark \ref{Remark: Renormalization}, the approximate
AH and PAM need to be renormalized in order to give a nontrivial $\eps\to0$ limit
in $d=2,3$. More specifically, as shown in \cite{HairerLabbe,HairerLabbe2,Labbe},
for every fixed $t>0$, one has
\begin{align}
\label{Equation: Renormalization}
\La_k(A,Q_t):=\lim_{\eps\to0}(\La_k(A_\eps,Q_t)-c_\eps)
\qquad\text{and}\qquad
U(t):=\lim_{\eps\to0}U_\eps(t)\mr e^{-tc_\eps},
\end{align}
where the renormalization constant $c_\eps$ blows up as $\eps\to0$
on the order of $|\log\eps|$ when $d=2$
and $\eps^{-1}$ when $d=3$ (up to constants and lower order terms).

Given that $\mf c_d<1/d$ by \eqref{Equation: Singular Phase Lower Constant},
when $d=2,3$ in the singular phase we always assume that $\eps(t)\gg(\log t)^{-1/(4-d)-1/d}$,
which implies that $c_{\eps(t)}\ll(\log t)^{2/(4-d)}$. Thus, the fact that we do not
include renormalization constants in Theorem \ref{Theorem: Main Singular} does
not contradict \eqref{Equation: Renormalization}. That said, if $\eps(t)$ vanishes
so quickly that $c_{\eps(t)}\gg(\log t)^{2/(4-d)}$, then it is not clear that we can expect
the asymptotics of
\[\frac{\La_k(A_{\eps(t)},Q_t)}{(\log t)^{2/(4-d)}}
\qquad\text{and}\qquad
\frac{\log U_{\eps(t)}(t)}{t\,(\log t)^{2/(4-d)}}\]
to be meaningful without the renormalizations $\La_k(A_{\eps(t)},Q_t)-c_{\eps(t)}$
and $U_{\eps(t)}(t)\mr e^{-tc_{\eps(t)}}$.

In light of this, one of the main insights of this paper is that
the asymptotics of the AH/PAM with WN can
be accessed even without requiring the use of renormalizations,
so long as $\eps(t)$ is not too big or small.
In fact, a comparison of Theorem \ref{Theorem: White Noise}
with known results for the AH/PAM with WN
provides quantitative upper bounds on the difference between the
latter and their smooth approximations for large $t$:
Among the main results of \cite{Chen14} (for $d=1$)
and \cite{ChoukVZ,KPZ} (for $d=2$) are the following:

\begin{theorem}
\label{Theorem: White Noise}
Let $A$ and $U(t)$ be the AH and PAM total mass with WN in $d=1,2$.
\[\lim_{t\to\infty}\frac{\La_k(A,Q_t)}{(\log t)^{2/(4-d)}}
=\mf L_d\qquad\text{in probability}\]
for every $k\in\mbb N$, and
\[\lim_{t\to\infty}\frac{\log U(t)}{t\,(\log t)^{2/(4-d)}}
=\mf L_d\qquad\text{in probability}.\]
\end{theorem}

By combining the above with Theorem \ref{Theorem: White Noise},
we obtain the following:

\begin{corollary}
\label{Corollary: Application}
Let $d=1,2$ and $\eps(t)$ be in the singular phase.
\begin{align}
\label{Equation: Quantitative Estimate 1}
\lim_{t\to0}\frac{|\La_k(A,Q_t)-\La_k(A_{\eps(t)},Q_t)|}{(\log t)^{2/(4-d)}}
=0\qquad\text{in probability}
\end{align}
for every $k\in\mbb N$, and
\begin{align}
\label{Equation: Quantitative Estimate 2}
\lim_{t\to0}\frac{\big|\log U(t)-\log U_{\eps(t)}(t)\big|}{t(\log t)^{2/(4-d)}}=0\qquad\text{in probability}.
\end{align}
\end{corollary}

In particular, if the estimates \eqref{Equation: Quantitative Estimate 1}
and \eqref{Equation: Quantitative Estimate 2} can be independently
established, then this would provide a new elementary proof of
Theorem \ref{Theorem: White Noise} in $d=2$.
It would also be interesting to see if similar results can be proved in $d=3$;
although to the best of our knowledge, an analog of Theorem \ref{Theorem: White Noise}
is not yet proved in this case.
We leave such questions open for future investigations.

\subsection{Organization}

The remainder of this paper is organized as follows.
In Section \ref{Section: Strategy}, we discuss
the strategy of proof for Theorems \ref{Theorem: Main Regular}
and \ref{Theorem: Main Singular}, including an intuitive explanation
of why the phase transition therein occurs at the critical rate
$(\log t)^{-1/(4-d)}$.
In Section \ref{Section: Setup}, we introduce the notation used in
our paper and state various classical results that lie
at the heart of our proof. In Section
\ref{Section: Eigenvalues} we prove the eigenvalue asymptotics
\eqref{Equation: Main Regular Eigenvalue} and \eqref{Equation: Main Singular Eigenvalue},
in Section \ref{Section: Quenched Proof} we prove the total mass asymptotics
\eqref{Equation: Main Regular TM} and \eqref{Equation: Main Singular TM},
and in Section \ref{Section: Annealed Proof}
we prove Theorem \ref{Theorem: Annealed}.

\section{Proof Strategy for Theorems \ref{Theorem: Main Regular} and \ref{Theorem: Main Singular}}
\label{Section: Strategy}

\subsection{PAM Total Mass Reduces to Leading Eigenvalue Asymptotics}

The main ingredient of the proofs of
\eqref{Equation: Main Regular TM} and
\eqref{Equation: Main Singular TM} consists
of the heuristic
\begin{align}
\label{Equation: Informal TM Principle}
U_{\eps(t)}(t)\approx\mr e^{t\La_1(A_{\eps(t)},Q_t)}\qquad\text{as }t\to\infty,
\end{align}
which completely reduces the quenched total mass
asymptotics to the exponential of the leading eigenvalue.
A rigorous version of this heuristic can be achieved by using
semigroup theory/the Feynman-Kac formula, as shown in
Sections \ref{Section: Semigroup and Feynman-Kac} and \ref{Section: Quenched Proof}.

\begin{remark}
As per \eqref{Equation: Informal TM Principle}, the
PAM asymptotics are only determined by the behavior of the
leading eigenvalue $\La_1(A_{\eps(t)},Q_t)$. That said, we nevertheless include a
statement for $\La_k(A_{\eps(t)},Q_t)$, $k\geq2$, in Theorems
\ref{Theorem: Main Regular} and \ref{Theorem: Main Singular} for the following reasons:
\begin{enumerate}
\item The asymptotics of $\La_k(A_{\eps(t)},Q_t)$ are of independent interest from
the point of view of the spectral theory of random Schr\"odinger operators; and
\item as we will show in Section \ref{Section: Reduction of Eigenvalue Asymptotics},
once the asymptotics of $\La_1(A_{\eps(t)},Q_t)$ are established, those of
$\La_k(A_{\eps(t)},Q_t)$ more or less immediately follow using a simple argument;
hence the additional statement does not require a more involved proof.
\end{enumerate}
\end{remark}

This type of argument for computing the asymptotics
of the total mass dates back to at least the work
of G\"artner and Molchanov \cite[Sections 2.4 and 2.5]{GartnerMolchanov2},
and was used in several more papers since then
(e.g., \cite{Chen12,Chen14,GartnerKonig,GartnerKonigMolchanov}).
The particular implementation of the argument used
in this paper most closely resembles that of
\cite{Chen14} (more specifically, see \cite[Sections 3 and 4]{Chen14}).
From the technical point of view, the argument deployed in
this paper is simultaneously simpler and more involved than
that of \cite{Chen14}: On the one hand, the fact that we do
not deal with noises that are Schwartz distributions allows to
sidestep a number of technical hurdles encountered in \cite{Chen14},
such as the approximation arguments
\cite[(2.22) and Sections 3, 4, and A.1]{Chen14}. On the other
hand, the need to consider a different noise (namely,
$\xi_{\eps(t)}$) for every value of $t$ and to distinguish
between two regimes of $\eps(t)$ increases the complexity
of some arguments.

\subsection{Eigenvalue Asymptotics}

We now discuss how the eigenvalue asymptotics \eqref{Equation: Main Regular Eigenvalue}
and \eqref{Equation: Main Singular Eigenvalue} are obtained. In Section
\ref{Section: Reduction of Eigenvalue Asymptotics}, we show that the eigenvalues
$\La_k(A_{\eps(t)},Q_t)$ for $k\geq2$ have the same asymptotics as
$\La_1(A_{\eps(t)},Q_t)$; hence we only need to prove asymptotics for the
leading eigenvalue.
By the min-max principle, we can write
\begin{align}
\label{Equation: Min-Max Informal}
\La_1(A_\eps,Q_t)=\sup_{\phi\in C_0^\infty(Q_t),~\|\phi\|_2=1}\big(\langle \xi_\eps,\phi^2\rangle-\tfrac12 \mc E(\phi)\big),
\end{align}
where $C_0^\infty(Q_t)$ denotes the set of smooth and compactly supported
functions on the box $Q_t$, and we use $\mc E$ as shorthand for the Dirichlet form induced by $\De$
(i.e., \eqref{Equation: Dirichlet Form}).
In this expression, we note that there is a competition between two terms:

On the one hand, maximizing $\langle \xi_\eps,\phi^2\rangle$
provides an incentive for $\phi$ to allocate all of its mass at the maximum of $\xi_\eps$ on $Q_t$.
More specifically, by the $L^1/L^\infty$ H\"older inequality, we have that
\[\sup_{\phi\in C_0^\infty(Q_t),~\|\phi\|_2=1}\langle \xi_\eps,\phi^2\rangle\leq\sup_{x\in Q_t}|\xi_\eps(x)|,\]
where the supremum over $\phi$ is achieved (at least formally) at any delta Dirac distribution $\de_{x_{\eps,t}}$
such that the point $x_{\eps,t}$ achieves $\xi_\eps$'s supremum on $Q_t$'s closure.
On the other hand, the term $-\frac12\mc E(\phi)$ penalizes functions with very substantial
variations, such as functions that are very close to a Dirac distribution. Thus, while we expect that the
eigenfunction that achieves the supremum in \eqref{Equation: Min-Max Informal} is localized near $\xi_{\eps}$'s
maximum on $Q_t$, its gradient cannot be too large.

Then, understanding the asymptotics of $\La_1(A_{\eps(t)},Q_t)$
is a matter of identifying the contributions of both of these effects in the large-$t$ limit. In this paper,
this is carried out using so-called localization bounds (see Section \ref{Section: Localization Bounds} for the details
as well as references identifying previous works where this idea has already appeared).
That is, we partition $Q_t$ into smaller sub-boxes $B_1(t),B_2(t),\ldots,B_n(t)$,
hoping that one of the $B_i(t)$'s will contain the bulk of the mass of the leading eigenfunction
(which itself is localized near $\xi_{\eps(t)}$'s maximizer on $Q_t$); hence
\begin{align}
\label{Equation: Localization Informal}
\La_1(A_{\eps(t)},Q_t)\approx\max_{1\leq i\leq n}\La_1\big(A_{\eps(t)},B_i(t)\big).
\end{align}
The task of understanding the eigenvalue asymptotics is now reduced to
\begin{enumerate}
\item identifying the size that the $B_i(t)$'s must have as $t\to\infty$
to capture the bulk of the mass of the leading eigenfunction (which relies
on understanding the tradeoff between $\langle \xi_\eps,\phi^2\rangle$
and $-\tfrac12\mc E(\phi)$ in \eqref{Equation: Min-Max Informal}); and
\item analyzing the asymptotics of the maximum on the right-hand side of \eqref{Equation: Localization Informal}.
As $\xi_{\eps(t)}$ is Gaussian, this relies on the extreme value theory of independent Gaussian fields
(as $\eps(t)\to0$, $\xi_{\eps(t)}$ becomes uncorrelated over very small distances).
\end{enumerate}

We now provide a heuristic based on a variety of previous results that serves as the main
guide in carrying out the above, that explains the asymptotics obtained in Theorems
\ref{Theorem: Main Regular} and \ref{Theorem: Main Singular}, and that explains why
a transition occurs at $\eps(t)\sim(\log t)^{-1/(4-d)}$.

\subsection{Eigenvalue Scaling Heuristics}

\subsubsection{Case $\eps(t)=1$}

The starting point of our heuristic is the second order asymptotics for the total
mass found by G\"artner, K\"onig, and Molchanov in \cite{GartnerKonigMolchanov}.
Combining Theorem 1.1 in that paper with the asymptotic equivalence
in \eqref{Equation: Informal TM Principle}, we have the following result, which
settles the special case where $\eps(t)=1$ in Theorem \ref{Theorem: Main Regular}:

\begin{theorem}[\cite{GartnerKonigMolchanov}]
Let us denote
\begin{align}
\label{Equation: ht}
L_t:=\sqrt{2dR(0)\log t}
\end{align}
and
\begin{align}
\label{Equation: chi}
l_t:=\frac{\mr{Tr}\big[\big(-R''(0)\big)^{1/2}\big]}{2}\left(\frac{2d}{R(0)}\log t\right)^{1/4},
\end{align}
where $R''$ denotes the Hessian matrix of the covariance.
As $t\to\infty$, one has
\begin{align}
\label{Equation: 2nd order}
\La_1(A_1,Q_t)=L_t-l_t+o(l_t).
\end{align}
\end{theorem}

On the one hand, the leading order term $L_t$ is solely
determined by the tendency of the eigenfunction to localize near
$\xi_1$'s maximum on $Q_t$. Indeed, the specific form
of \eqref{Equation: ht} is due to the
classical
extreme value theory of Gaussian processes:

\begin{lemma}
\label{Lemma: LLN Gaussian Maxima}
Suppose that $X$ is a centered stationary Gaussian process on $\mbb R^d$
or $\mbb Z^d$, assuming that $X$ has continuous sample paths if it is on $\mbb R^d$.
If the covariance
$\mbf E[X(0)X(x)]$ of $X$ vanishes as $|x|_2\to\infty$,
then
\begin{align}
\label{Equation: Gaussian Log Supremum}
\lim_{t\to\infty}\sup_{x\in Q_t}\frac{X(x)}{\sqrt{\log t}}=\sqrt{2d\mbf E[X(0)^2]}\qquad\text{almost surely.}
\end{align}
\end{lemma}
\begin{proof}
We refer to \cite[Section 2.1]{CarmonaMolchanov} and references therein
for \eqref{Equation: Gaussian Log Supremum} in the continuous case.
For the discrete case, we point to \cite[Theorem 3.4]{Pickands}
(\cite{Pickands} is only stated in $d=1$,
but its argument can adapted to $d\geq2$ with only trivial modifications).
\end{proof}

On the other hand, as explained in \cite[Section 1.6]{GartnerKonigMolchanov},
the second order term $-l_t$ comes from the contribution
of $-\frac12\mc E(\phi)$. More specifically, the leading eigenfunction will localize
in a box of approximate size
\begin{align}
\label{Equation: st}
s_t:=(2dR(0)\log t)^{-1/2},
\end{align}
and the trace in \eqref{Equation: chi} contains information
about the local geometry of the leading eigenfunction near its maximum.

In summary, if $\eps(t)=1$, then $\langle\xi_{1},\phi^2\rangle$ dominates
$-\frac12\mc E(\phi)$; hence the only contribution in the first order asymptotics
comes from $\xi_{1}$'s maximum over $Q_t$.

\subsubsection{Subcritical Phase}

The general statement of \eqref{Equation: Main Regular Eigenvalue}
and \eqref{Equation: Main Singular Eigenvalue}
can be seen as an extension of the argument for $\eps(t)=1$
to its most general incarnation in the setting of asymptotically singular noise:
By a straightforward rescaling (i.e., \eqref{Equation: L2 Rescaling}),
it can be seen that for every $t>0$,
$\La_1(A_{\eps(t)},Q_t)$ is equal to the leading Dirichlet
eigenvalue of the operator $\eps(t)^{-2}(\tfrac12\De+\eps(t)^{(4-d)/2}\xi_1)$
on the box $Q_{t/\eps(t)}$. If we apply this rescaling to the
quantities in \eqref{Equation: 2nd order} (i.e., replace $R$ by $\eps(t)^{(4-d)}R$ and scale space by $\eps(t)$), then this suggests the
following: If we denote
\begin{align}
\label{Equation: ht 2}
\tilde L_t&:=\eps(t)^{(4-d)/2}\sqrt{2dR(0)\log t},\\
\label{Equation: chi 2}
\tilde l_t&:=\eps(t)^{(4-d)/4}\frac{\mr{Tr}\big[\big(-R''(0)\big)^{1/2}\big]}{2}\left(\frac{2d}{R(0)}\log t\right)^{1/4},\\
\label{Equation: st 2}
\tilde s_t&=\eps(t)^{-(4-d)/2}(2dR(0)\log t)^{-1/2}\cdot\eps(t),
\end{align}
then we expect from \eqref{Equation: 2nd order} that
\[\La_1(A_{\eps(t)},Q_t)
=\eps(t)^{-2}\big(\tilde L_t-\tilde l_t+o(\tilde l_t)\big),\]
where $\tilde L_t$ is determined by the maximum of $\xi_{\eps(t)}$ on $Q_t$,
$-\tilde l_t$ comes from the contribution of $-\frac12\mc E(\phi)$,
and the leading eigenfunction is localized near this maximum in a box of size
$\tilde s_t$. Then, given that for $d=1,2,3$,
\begin{align}
\label{Equation: Phase Transition Explanation 0}
\eps(t)^{(4-d)/2}\sqrt{\log t}\gg\eps(t)^{(4-d)/4}(\log t)^{1/4}
\end{align}
if and only if $\eps(t)\gg(\log t)^{-1/(4-d)}$, we infer that the subcritical phase
coincides precisely with the regime where $\tilde L_t\gg\tilde l_t$,
that is, the regime where
the leading order asymptotics of $\La_1(A_{\eps(t)},Q_t)$ are solely determined by
the maximum of $\xi_{\eps(t)}$ on $Q_t$.
At this point, by noting that
\[\eps(t)^{-2}\tilde L_t=\eps(t)^{-d/2}\sqrt{\log t}\sqrt{2d R(0)}\big(1+o(1)\big)\qquad\text{as }t\to\infty,\]
we recover the subcritical eigenvalue asymptotics claimed in \eqref{Equation: Main Regular Eigenvalue}.

\begin{remark}
\label{Remark: Phase Transition Explanation}
The fact that this argument relies crucially
on \eqref{Equation: Phase Transition Explanation 0} also explains
why there is no phase
transition when $d\geq4$ in Theorems \ref{Theorem: Main Regular}
and \ref{Theorem: Main Singular}: Indeed,
if $d\geq4$, then \eqref{Equation: Phase Transition Explanation 0}
always holds, and thus we expect that the leading order asymptotics of $\La_1(A_{\eps(t)},Q_t)$
should be determined by the maximum of $\xi_{\eps(t)}$ on $Q_t$ no matter how small
$\eps(t)$ is. See Theorem \ref{Theorem: d4} and
Sections \ref{Section: Remark on d4 1} and \ref{Section: Remark on d4 2} for the details.
\end{remark}

\subsubsection{Supercritical Phase}

If $\eps(t)$ vanishes faster than the critical threshold $(\log t)^{-1/(4-d)}$,
then \eqref{Equation: Phase Transition Explanation 0} is no longer true.
In particular, the terms $\tilde L_t$ and $\tilde l_t$
in \eqref{Equation: ht 2} and \eqref{Equation: chi 2} coalesce, and thus the leading
order asymptotics of $\La_1(A_{\eps(t)},Q_t)$ can no longer be expected to be
solely determined by the maximum of $\xi_{\eps(t)}$ on $Q_t$.

We expect that the asymptotics of $\La_1(A_{\eps(t)},Q_t)$ should in some sense
stabilize to that of $\La(A,Q_t)$ when $\eps(t)$ is very small. Thus, it is natural to hypothesize
that the magnitude of the terms corresponding to $\eps(t)^{-2}(\tilde L_t-\tilde l_t)$ and $\tilde s_t$ in this
regime can be obtained by replacing $\eps(t)$ in \eqref{Equation: ht 2}--\eqref{Equation: st 2}
by the value of the critical threshold $a(t):=(\log t)^{-1/(4-d)}$. Following this hypothesis, we expect that,
up to a constant,
\begin{align}
\label{Equation: 2nd order 3}
\La_1(A_{\eps(t)},Q_t)\asymp a(t)^{-2}\cdot a(t)^{(4-d)/2}\sqrt{\log t}=(\log t)^{2/(4-d)},
\end{align}
and that the corresponding eigenfunction localizes in a box of size (up to a constant)
\begin{align}
\label{Equation: st 3}
a(t)^{-(4-d)/2}(\log t)^{-1/2}\cdot a(t)=(\log t)^{-1/(4-d)}.
\end{align}
This heuristic is further corroborated by the facts that
\begin{enumerate}
\item The known asymptotics for the AH/PAM
with WN found in \cite{Chen14} for $d=1$
and \cite{ChoukVZ,KPZ} for $d=2$ correspond to \eqref{Equation: 2nd order 3} (c.f., Theorem \ref{Theorem: White Noise}); and
\item in the one-dimensional case \cite{Chen14}, Chen showed that the asymptotics in question
can be obtained with a localization argument of the form \eqref{Equation: Localization Informal}
with sub-boxes $B_i(t)$ of size $(\log t)^{-1/3}$ (up to a constant) as $t\to\infty$.
\end{enumerate}

From the technical standpoint (see the outline in steps (1)--(3) in \cite[Page 583]{Chen14}),
a crucial innovation
in \cite{Chen14} lies in relating
the maxima on the right-hand side of \eqref{Equation: Localization Informal}
for WN in $d=1$ to the
extreme value theory of the function-valued
Gaussian process
\[\phi\mapsto\langle\phi,\xi\rangle,\qquad \phi:\mbb R^d\to\mbb R\]
on carefully chosen function spaces that are amenable to computation.
Thus, one of the insights of this paper is that this type of argument can be
extended in $d=2,3$, so long
as we consider asymptotically singular noise $\xi_{\eps(t)}$
where $\eps(t)$ is not too large (or too small when $d=2,3$).

\section{Setup and Notation}
\label{Section: Setup}

\subsection{Basic Notations}

\begin{definition}
Given $1\leq p\leq\infty$,
we use $\|f\|_p$ to denote the
$L^p$ norm of a function $f:\mbb R^d\to\mbb R$,
and $|x|_p$ to denote the
$\ell^p$-norm of a vector $x\in\mbb R^d$.
We use $\langle f,g\rangle$
to denote the Euclidean inner product on $L^2(\mbb R^d)$, and
$f*g$ to denote the convolution.
\end{definition}

\begin{definition}
Given an open set $\Om\subset\mbb R^d$, we use $C_0^\infty(\Om)$
to denote the set of smooth functions $\phi:\mbb R^d\to\mbb R$ with
compact support in $\Om$. We denote the Dirichlet form of such functions as
\begin{align}
\label{Equation: Dirichlet Form}
\mc E(\phi):=\int_{\mbb R^d}|\nabla\phi(x)|_2^2\d x,
\end{align}
where $\nabla$ denotes the gradient, and we denote the function spaces
\begin{align}
\nonumber
S(\Om)&:=\{\phi\in C_0^\infty(\Om):\|\phi\|_2=1\},\\
\label{Equation: W Function Space}
W(\Om)&:=\{\phi\in C_0^\infty(\Om):\|\phi\|_2^2+\tfrac12 \mc E(\phi)=1\}.
\end{align}
\end{definition}

\begin{definition}
For every $z\in\mbb R^d$, we define the translation operator
\begin{align}
\label{Equation: Translation Operator}
\tau_z\phi(x):=\phi(x-z),\qquad \phi\in C_0^\infty(\mbb R^d).
\end{align}
For every $\eta>0$ and $\phi\in C_0^\infty(\mbb R^d)$,
we define the rescaled function
\begin{align}
\label{Equation: L2 Rescaling}
\phi^{(\eta)}(x):=\eta^{d/2}\phi(\eta x),\qquad x\in\mbb R^d.
\end{align}
\end{definition}

\begin{remark}
\label{Remark: L2 Rescaling}
It is easy to see that for every domain $\Om\subset\mbb R^d$ and $\eta>0$,
\begin{enumerate}
\item $\phi\in S(\eta\Om)$ if and only if $\phi^{(\eta)}\in S(\Om)$, and
\item $ \mc E(\phi^{(\eta)})
=\eta^2 \mc E(\phi)$.
\end{enumerate}
\end{remark}

\subsection{Covariance Semi Inner Product}

\begin{definition}
For every $\eps>0$, we denote the covariance semi inner product by
\[\langle f,g\rangle_{R_\eps}:=\langle f*\bar R_\eps,g*\bar R_\eps\rangle
=\int_{(\mbb R^d)^2}f(x)R_\eps(x-y)g(y)\d x\dd y\]
for $f,g:\mbb R^d\to\mbb R$, and we denote the associated seminorm by
\[\|f\|_{R_\eps}:=\sqrt{\langle f,f\rangle_{R_\eps}}.\]
In particular, since $R_\eps\to\de_0$, one has
\begin{align}
\label{Equation: Continuous Limit of Forms}
\lim_{\eps\to0}\langle\phi,\psi\rangle_{R_\eps}=\langle\phi,\psi\rangle,
\qquad\phi,\psi\in C_0^\infty(\mbb R^d).
\end{align}
\end{definition}

\subsection{Operator Theory and Localization Bounds}

\label{Section: Localization Bounds}

For every $\eps,\si>0$, let us denote the operator
\begin{align}
A^{(\si)}_\eps:=\tfrac12\De+\si\xi_\eps,
\end{align}
so that, in particular, $A_\eps=A^{(1)}_\eps$.
For any bounded and connected open set $\Om\subset\mbb R^d$,
the operator $-A^{(\si)}_\eps$
with Dirichlet boundary conditions on $\Om$
is self-adjoint on $L^2(\Om)$ and has compact resolvent,
and $C_0^\infty(\Om)$ is a core for its
quadratic form:
\[-\langle \phi, A^{(\si)}_\eps\phi\rangle:=-\si\langle \xi_\eps,\phi^2\rangle+\tfrac12 \mc E(\phi),\qquad\phi\in C_0^\infty(\Om)\]
(e.g., \cite[Example 3.16.4]{SimonBook};
we refer more generally to \cite[Section 7.5]{SimonBook} for the
operator-theoretic terminology used here).
In particular, it follows from the min-max principle
(e.g., \cite[Theorem 7.8.10]{SimonBook}) that
\begin{align}
\La_k(A^{(\si)}_\eps,\Om)
\label{Equation: Courant-Fischer}
=\sup_{\substack{\phi_1,\ldots,\phi_k\in C_0^\infty(\Om)\\\langle\phi_i,\phi_j\rangle=0~\forall i\neq j}}
~\inf_{\substack{\phi\in\mr{span}(\phi_1,\ldots,\phi_k)\\\|\phi\|_2=1}}
\big(\si\langle \xi_\eps,\phi^2\rangle-\tfrac12 \mc E(\phi)\big),\qquad k\in\mbb N,
\end{align}
with matching eigenfunctions
forming an orthonormal basis of $L^2(\Om)$.
The following localization bounds are the main technical tools in our analysis of the
asymptotics of the $A^{(\si)}_\eps$'s eigenvalues:

\begin{lemma}\label{Lemma: Localization Lower Bound}
For every $\eps>0$ and bounded open sets $\Om\subset\mbb R^d$, $\Om_1,\ldots,\Om_n\subset\Om$,
\[\La_1(A^{(\si)}_\eps,\Om)\geq\max_{i=1,\ldots,n}\La_1(A^{(\si)}_\eps,\Om_i).\]
\end{lemma}

\begin{lemma}\label{Lemma: Localization Upper Bound}
There exists a constant $C>0$
such that for every $\eps,\ka>0$ and $r>\ka$, if we let
$Z:=2\ka\mbb Z_d\cap Q_r$,
then
\begin{align*}
\La_1(A^{(\si)}_\eps,Q_r)
\leq\frac{C}{\ka}+\max_{z\in Z}
\La_1(A^{(\si)}_\eps,z+Q_{\ka+1}).
\end{align*}
\end{lemma}

The use of localization bounds such as
Lemmas \ref{Lemma: Localization Lower Bound} and \ref{Lemma: Localization Upper Bound}
date back to at least to work of G\"artner, K\"onig, and Molchanov
\cite{GartnerKonig,GartnerKonigMolchanov}.
While Lemma \ref{Lemma: Localization Lower Bound} is a trivial
consequence of the min-max principle \eqref{Equation: Courant-Fischer}, 
Lemma \ref{Lemma: Localization Upper Bound} is more delicate;
we refer to the proof of \cite[(2.27)]{Chen14} for the latter.

\subsection{Semigroup Theory and Feynman-Kac Formula}
\label{Section: Semigroup and Feynman-Kac}

\begin{definition}
For every open set
$\Om\subset\mbb R^d$, we let $T_\Om$
denote the first exit time of $\Om$ by the Brownian motion
$B$, that is, $T_\Om:=\inf\{t\geq 1:B(t)\not\in\Om\}.$

For every $x,y\in\mbb R^d$
and $t>0$, we use $\mbf E^x$ to denote the expectation
with respect to the law of the Brownian motion $\big(B\big|B(0)=x\big)$, and $\mbf E^{x,y}_t$ to denote the expectation
with respect to the law of the Brownian bridge $\big(B\big|B(0)=x\text{ and }B(t)=y\big)$,
both conditional on $\xi_\eps$.

We use $\ms G_t$ to denote the Gaussian kernel, that is,
\begin{align}
\label{Equation: Gaussian Kernel}
\ms G_t(x):=\frac{\mr e^{-|x|_2^2/2t}}{(2\pi t)^{d/2}},\qquad t>0,~x\in\mbb R^d.
\end{align}
\end{definition}

Let $\Om\subset\mbb R^d$ be a bounded open set.
The semigroup of the operator $A^{(\si)}_\eps$ with Dirichlet boundary conditions on $\Om$
is defined as the family of operators
\begin{align}
\label{Equation: Semigroup Def}
\mc T^{A^{(\si)}_\eps,\Om}_tf:=\sum_{k=1}^{\infty}\mr e^{t\La_k(A^{(\si)}_\eps,\Om)}\big\langle\Psi_k(A^{(\si)}_\eps,\Om),f\big\rangle\,\Psi_k(A^{(\si)}_\eps,\Om)
,\qquad t>0
\end{align}
acting on $f\in L^2(\Om)$,
where $\Psi_k(A^{(\si)}_\eps,\Om)$ ($k\in\mbb N$) denote the orthonormal eigenfunctions
associated with $\La_k(A^{(\si)}_\eps,\Om)$.
According to the Feynman-Kac formula (e.g.,  \cite[(34) and Theorem 3.27]{ChungZhao}),
$\mc T^{A^{(\si)}_\eps,\Om}_t$ is an integral operator on
$L^2(\Om)$ with kernel
\begin{align}
\label{Equation: Domain Feynman-Kac}
\mc T^{A^{(\si)}_\eps,\Om}_t(x,y)
:=\ms G_t(x-y)\,\mbf E^{x,y}_t\left[\exp\left(\si\int_0^t \xi_\eps\big(B(s)\big)\d s\right);T_\Om\geq t\right]
\end{align}
for all $t>0$ and $x,y\in\mbb R^d$, where, for any random variable $Y$ and event $E$, we denote
$\mbf E[Y;E]:=\mbf E[\mbf 1_E\,Y]$.
Since $\Om$ is bounded, the semigroup $(\mc T^{A^{(\si)}_\eps,\Om}_t)_{t>0}$
is Hilbert-Schmidt/trace class, and for every $t>0$
(e.g., \cite[Theorem 3.17]{ChungZhao}),
\begin{align}
\label{Equation: Trace Formula}
\sum_{k=1}^{\infty}\mr e^{t\La_k(A^{(\si)}_\eps,\Om)}=\mr{Tr}\left[\mc T^{A^{(\si)}_\eps,\Om}_t\right]
=\int_{\Om}\mc T^{A^{(\si)}_\eps,\Om}_t(x,x)\d x
=\left\|\mc T^{A^{(\si)}_\eps,\Om}_{t/2}\right\|_2^2<\infty.
\end{align}

One of the main ingredients in the proof of Theorems
\ref{Theorem: Main Regular} and \ref{Theorem: Main Singular}
consists of the observation that if $\Om$ is very large and contains the origin,
then we expect that
\[U_\eps(t)=\mbf E^0\left[\exp\left(\int_0^t\xi_\eps\big(B(s)\big)\d s\right)\right]
\approx\mbf E^0\left[\exp\left(\int_0^t\xi_\eps\big(B(s)\big)\d s\right);T_\Om\geq t\right].\]
Thanks to \eqref{Equation: Semigroup Def} and \eqref{Equation: Domain Feynman-Kac},
this then creates a connection between the asymptotics of $U_\eps(t)$ and that of
$\La_1(A_\eps,Q_t)$ as $t\to\infty$, allowing to formalize
the heuristic \eqref{Equation: Informal TM Principle}. In order to make this precise, we use the following
two technical results, which are the statements of \cite[(4.2) and (4.5)]{Chen12}
and \cite[(4.3) and (4.6)]{Chen12}, respectively
(and which are proved using a variety of time-cutoff arguments on
sub-intervals $[0,\eta]\subset[0,t]$ and H\"older's inequality with $p,q\geq1$):

\begin{proposition}
\label{Proposition: Quenched Upper Bound Technical}
Let us denote, for every $\eps,\si>0$ and $t\geq0$,
the quantity
\[U^{(\si)}_\eps(t):=\mbf E^0\left[\exp\left(\si\int_0^t\xi_\eps\big(B(s)\big)\d s\right)\right].\]
Let $r>0$ and $p,q>1$ be such that $1/p+1/q=1$.
For every $t\geq 1$ and $0<\eta<t$,
it holds that
\begin{multline}
\label{Equation: Quenched Upper Bound 1}
\mbf E^0\left[\exp\left(\int_0^t \xi_\eps\big(B(s)\big)\d s\right);T_{Q_r}\geq t\right]
\leq
U^{(q)}_\eps(\eta)^{1/q}\\
\cdot\left(\frac{1}{(2\pi\eta)^{d/2}}\int_{Q_r}\mbf E^x\left[\exp\left(p\int_0^{t-\eta} \xi_\eps\big(B(s)\big)\d s\right);T_{Q_r}\geq t-\eta\right]\d x\right)^{1/p}
\end{multline}
and for every $\tilde t\geq1$ and $\theta>0$,
\begin{align}
\label{Equation: Quenched Upper Bound 2}
\int_{Q_r}\mbf E^x\left[\exp\left(\theta\int_0^{\tilde t} \xi_\eps\big(B(s)\big)\d s\right);T_{Q_r}\geq \tilde t\right]\d x
\leq(2r)^d\mr e^{\tilde t\La_1(A^{(\theta)}_\eps,Q_r)}.
\end{align}
\end{proposition}

\begin{proposition}
\label{Proposition: Quenched Lower Bound Technical}
Let $r>0$ and
$p,q>1$ be such that $1/p+1/q=1$. For every $t\geq 1$ and $0<\eta<t$,
it holds that
\begin{multline}
\label{Equation: Quenched Lower Bound 1}
U_\eps(t)\\\geq U^{(-q/p)}_\eps(\eta)^{-p/q}
\left(\int_{Q_r}\ms G_\eta(x)\,
\mbf E^x\left[\exp\left(\frac 1p\int_0^{t-\eta} \xi_\eps\big(B(s)\big)\d s\right);T_{Q_r}\geq t-\eta\right]\d x\right)^p.
\end{multline}
For every $\tilde t>0$, $0<\eta<\tilde t$, and $\theta>0$, it holds that
\begin{multline}
\label{Equation: Quenched Lower Bound 2}
\int_{Q_r}\mbf E^x\left[\exp\left(\theta\int_0^{\tilde t} \xi_\eps\big(B(s)\big)\d s\right);T_{Q_r}\geq \tilde t\right]\d x
\\\geq(2\pi)^{pd/2}\eta^{d/2}\tilde t^{pd/2q}(2r)^{-2p/q}
\mr e^{-\eta(p/q)\La_1(A^{(q\theta/p)}_\eps,Q_r)}\mr e^{p(\tilde t+\eta)\La_1(A^{(\theta/p)}_\eps,Q_r)}.
\end{multline}
\end{proposition}

\section{Eigenvalue Asymptotics}
\label{Section: Eigenvalues}

Our purpose in this section is to prove
the eigenvalue asymptotics in Theorems
\ref{Theorem: Main Regular} and \ref{Theorem: Main Singular},
namely,
\eqref{Equation: Main Regular Eigenvalue} and \eqref{Equation: Main Singular Eigenvalue}.
For this purpose, in this section the main result we prove is the following:

\begin{theorem}
\label{Theorem: Eigenvalue Technical}
Let $w:[0,\infty)\to\mbb R^d$ be an arbitrary function.
Let $\theta,\al>0$ and $\be\geq0$ be fixed constants,
and define
\begin{align}
\label{Equation: General Radius}
r(t):=\begin{cases}
\theta t^\al\left(\eps(t)^{-d/2}\sqrt{\log t}\right)^\be&\text{if $\eps(t)$ is in the regular phase,}
\vspace{5pt}\\
\theta t^\al\left((\log t)^{2/(4-d)}\right)^\be&\text{if $\eps(t)$ is in the singular phase.}
\end{cases}
\end{align}
If $\eps(t)$ is in the regular phase, then for every $\si>0$,
\begin{align}
\label{Equation: Regular Eigenvalue}
\lim_{t\to\infty}\frac{\La_1\big(A^{(\si)}_{\eps(t)},w(t)+Q_{r(t)}\big)}{\eps(t)^{-d/2}\sqrt{\log r(t)}}&=\si\sqrt{2dR(0)}&\text{in probability.}
\end{align}
If $\eps(t)$ is in the singular phase, then for every $\si>0$,
\begin{align}
\label{Equation: Singular Eigenvalue}
\lim_{t\to\infty}\frac{\La_1\big(A^{(\si)}_{\eps(t)},w(t)+Q_{r(t)}\big)}{(\log r(t))^{2/(4-d)}}&=
\si^{4/(4-d)}\mf L_d&\text{in probability.}
\end{align}
\end{theorem}

\begin{remark}
Since $\xi_{\eps(t)}$ is translation invariant, if Theorem
\ref{Theorem: Eigenvalue Technical} holds for $w(t)=0$,
then it immediately follows that the same is also true for any other choice of $w(t)$.
We nevertheless state Theorem \ref{Theorem: Eigenvalue Technical} with
general $w(t)$, since this statement is used in the proof of the eigenvalue asymptotics
of $\La_k(A_{\eps(t)},Q_t)$ for $k\geq2$ (see Section \ref{Section: Reduction of Eigenvalue Asymptotics}).
\end{remark}

\begin{remark}
Following-up on the previous remark,
we note that apart from the presence of $w(t)$,
the statement of Theorem \ref{Theorem: Eigenvalue Technical}
has several differences with
\eqref{Equation: Main Regular Eigenvalue} and \eqref{Equation: Main Singular Eigenvalue},
making it simultaneously more and less general than the latter.

Firstly, Theorem \ref{Theorem: Eigenvalue Technical} is less general,
since it only concerns the leading eigenvalue $\La_1$.
This is due to the fact that the localization bounds in Lemmas
\ref{Lemma: Localization Lower Bound} and \ref{Lemma: Localization Upper Bound},
which are the main technical tools with which we prove Theorem \ref{Theorem: Eigenvalue Technical},
only apply to the first eigenvalue.
The fact that
\eqref{Equation: Main Regular Eigenvalue} and \eqref{Equation: Main Singular Eigenvalue}
follow from
\eqref{Equation: Regular Eigenvalue} and \eqref{Equation: Singular Eigenvalue}
comes from one aspect of Theorem \ref{Theorem: Eigenvalue Technical}
that is more general, namely, that we consider the asymptotics of
the leading eigenvalue on off-centered boxes $w(t)+Q_{r(t)}$ with side length
$2r(t)$ instead of $2t$.

Secondly, Theorem \ref{Theorem: Eigenvalue Technical} is more general
in the sense that we consider boxes of side length $r(t)$, as well as the
scaling factor $\si$ (which is equal to $1$ in
\eqref{Equation: Main Regular Eigenvalue} and \eqref{Equation: Main Singular Eigenvalue})
These more general aspects are used in the proof of the
total mass asymptotics in Theorems
\ref{Theorem: Main Regular} and \ref{Theorem: Main Singular};
we refer to Section \ref{Section: Quenched Proof} for the details.
\end{remark}

The remainder of this section is organized as follows.
In Section \ref{Section: Reduction of Eigenvalue Asymptotics},
we use Theorem \ref{Theorem: Eigenvalue Technical} to prove
\eqref{Equation: Main Regular Eigenvalue} and \eqref{Equation: Main Singular Eigenvalue}.
In Section \ref{Section: Proof of Eigenvalue Technical}, we prove
Theorem \ref{Theorem: Eigenvalue Technical}.

\subsection{Proof of
\eqref{Equation: Main Regular Eigenvalue} and \eqref{Equation: Main Singular Eigenvalue}}
\label{Section: Reduction of Eigenvalue Asymptotics}

We only prove the eigenvalue asymptotics
in the regular phase, as the proof in the singular
case follows from the same argument.
We begin with the following statement
(see, e.g., \cite[Theorem 8.5]{ChoukVZ}),
which connects the non-leading eigenvalues to $\La_1$:

\begin{lemma}
\label{Lemma: Reduction Lower Bound}
Let $k\in\mbb N$ and $t\geq 1$ be fixed.
If $z_1,\ldots,z_k\in\mbb R^d$ and $\ka>0$ are such that
$z_i+Q_\ka\subset Q_t$ for every $1\leq i\leq k$,
and
\[\langle\mbf 1_{z_i+Q_\ka},\mbf 1_{z_j+Q_\ka}\rangle=0\]
for every $1\leq i<j\leq k$, then for every $\eps>0$
one has
\[\La_k(A_\eps,Q_t)\geq\min_{1\leq i\leq k}\La_1(A_\eps;z_i+Q_\ka).\]
\end{lemma}

We also make the following simple remark:

\begin{remark}
\label{Remark: General Scaling Function}
$r(t)$ in \eqref{Equation: General Radius}
satisfies the following as $t\to\infty$:
\begin{align}
\label{Equation: log r(t) regular}
\eps(t)^{-d/2}\sqrt{\log r(t)}&=\sqrt{\al}\,\eps(t)^{-d/2}\sqrt{\log t}\,\big(1+o(1)\big)&\text{regular phase},\\
\label{Equation: log r(t) singular}
\big(\log r(t)\big)^{2/(4-d)}&=\al^{2/(4-d)}(\log t)^{2/(4-d)}\,\big(1+o(1)\big)&\text{singular phase}.
\end{align}
\end{remark}

Thanks to Theorem \ref{Theorem: Eigenvalue Technical}
(with $\al=1$, $\be=0$, and $\si=1$)
and Remark \ref{Remark: General Scaling Function},
every sequence of $t>0$ such that $t\to\infty$ has a subsequence $(t_n)_{n\in\mbb N}$
along which
\[\lim_{n\to\infty}\frac{{\La_1\big(A_{\eps(t_n)};w(t_n)+Q_{\theta t_n}\big)}}
{\eps(t_n)^{-d/2}\sqrt{\log t_n}}=\sqrt{2dR(0)}\qquad\text{almost surely}\]
for every choice of $\theta\in(0,1]\cap\mbb Q$ and $w(t_n)\in\mbb Q$.

On the one hand, by taking $w(t_n)=0$
and $\theta=1$
and using the trivial inequality $\La_k(A_\eps,\Om)\leq\La_1(A_\eps,\Om)$
for all $k\in\mbb N$ and $\Om\subset\mbb R^d$, we get
\begin{align}
\label{Equation: Higher Eigenvalues Upper Bound}
\limsup_{n\to\infty}\frac{{\La_k(A_{\eps(t_n)},Q_{t_n})}}
{\eps(t_n)^{-d/2}\sqrt{\log t_n}}\leq\sqrt{2dR(0)}\qquad\text{almost surely}.
\end{align}
On the other hand, we can find a small enough constant $\theta\in(0,1]\cap\mbb Q$ (that only depends on $k$)
such that for every $t\geq 0$, there exists $w^1(t),\ldots,w^k(t)\in\mbb Q$ such that
the sets $w^i(t)+Q_{\theta t}$ are mutually disjoint and inside $Q_t$.
Consequently, up to taking a further subsequence of $(t_n)_{n\in\mbb N}$
in \eqref{Equation: Higher Eigenvalues Upper Bound},
we obtain from Theorem \ref{Theorem: Eigenvalue Technical},
Remark \ref{Remark: General Scaling Function},
and Lemma \ref{Lemma: Reduction Lower Bound} that
\[\liminf_{n\to\infty}\frac{\La_k(A_{\eps(t_n)},Q_{t_n})}{\eps(t_n)^{-d/2}\sqrt{\log t_n}}
\geq\min_{1\leq i\leq k}\liminf_{n\to\infty}\frac{\La_1\big(A_{\eps(t_n)};w^i(t_n)+Q_{\theta t_n}\big)}{\eps(t_n)^{-d/2}\sqrt{\log t_n}}=\sqrt{2dR(0)}\]
almost surely.
Combined with \eqref{Equation: Higher Eigenvalues Upper Bound},
this completes the proof of the asymptotic
for $\La_k(A_{\eps(t)},Q_t)$
in Theorem \ref{Theorem: Main Regular}.

\subsection{Proof of Theorem \ref{Theorem: Eigenvalue Technical}}
\label{Section: Proof of Eigenvalue Technical}

We now prove Theorem \ref{Theorem: Eigenvalue Technical}.
Given that $\xi_\eps$ is stationary, there is no
loss of generality in assuming that $w(t)=0$; hence we
need only prove asymptotics for $\La_1(A^{(\si)}_{\eps(t)},Q_{r(t)})$.
We separate the proof of Theorem \ref{Theorem: Eigenvalue Technical} into
four steps, namely, matching lower and upper bounds for
\eqref{Equation: Regular Eigenvalue} and \eqref{Equation: Singular Eigenvalue}:

\subsubsection{Step 1. Lower Bound for \eqref{Equation: Regular Eigenvalue}}
\label{Section: Eigenvalues Regular Phase Lower}

Let $\eps(t)$ be in the regular phase.
Let $\ka>0$ be large enough so that $R$
is supported on $Q_{\ka/2}$.
For large $t>0$, let us define
\begin{align}\label{Equation: Regular Lower Bound Z}
Z_t:=3\ka\eps(t)\mbb Z^d\cap Q_{r(t)-\ka\eps(t)}.
\end{align}
By Lemma \ref{Lemma: Localization Lower Bound}
(with the sets $\Om_i$ given by $z+Q_{\ka\eps(t)}$ for all $z\in Z_t$), we have
\begin{align}\label{Equation: Regular Localization Lower Bound 1}
\nonumber
\La_1(A^{(\si)}_{\eps(t)},Q_{r(t)})&\geq
\max_{z\in Z_t}\sup_{\phi\in S(z+Q_{\ka\eps(t)})}
\big(\si\big\langle \xi_{\eps(t)},\phi^2\big\rangle-\tfrac12 \mc E(\phi)\big)\\
&=\sup_{\phi\in S(Q_{\ka\eps(t)})}
\left(\max_{z\in Z_t}\si\big\langle \xi_{\eps(t)},(\tau_z\phi)^2\big\rangle-\tfrac12 \mc E(\phi)\right),
\end{align}
where we recall that $\tau_z$ is the translation operator defined in
\eqref{Equation: Translation Operator},
and we used in \eqref{Equation: Regular Localization Lower Bound 1}
the fact that $ \mc E$ is translation invariant.
The set of functions $\phi$ over which the supremum
\eqref{Equation: Regular Localization Lower Bound 1} is taken depends on $t$.
When considering the large-$t$ limit of this expression, however,
it is more convenient
to consider $\phi$ from a fixed function space. For this reason, we consider
rescaled functions:
By Remark \ref{Remark: L2 Rescaling} (with $\eta=\eps(t)^{-1}$),
we find that for every $\phi\in S(Q_\ka)$, one has
\begin{align}\label{Equation: Regular Localization Lower Bound 2}
\La_1(A^{(\si)}_{\eps(t)},Q_{r(t)})\geq
\left(\max_{z\in Z_t}\si\big\langle \xi_{\eps(t)},(\tau_z\phi^{(1/\eps(t))})^2\big\rangle\right)
-\eps(t)^{-2}\tfrac12 \mc E(\phi).
\end{align}
Until further notice, we assume that we are considering a single fixed
function $\phi\in S(Q_\ka)$.
By a straightforward change of variables,
\begin{align}
\label{Equation: Regular Localization Lower Bound 3}
\int_{\mbb R^d}\xi_{\eps(t)}(x)\eps(t)^{-d}\phi\left(\frac{x-z}{\eps(t)}\right)^2\d x=
\eps(t)^{-d/2}\int_{\mbb R^d}\xi_1(x)\phi\left(x-\frac z{\eps(t)}\right)^2\d x,
\end{align}
and thus \eqref{Equation: Regular Localization Lower Bound 2} yields
\begin{align}
\label{Equation: Regular Localization Lower Bound 4}
\La_1(A^{(\si)}_{\eps(t)},Q_{r(t)})\geq
\left(\si\eps(t)^{-d/2}\max_{z\in Z_t}\big\langle \xi_1,(\tau_{z/\eps(t)}\phi)^2\big\rangle\right)
-\eps(t)^{-2}\tfrac12 \mc E(\phi).
\end{align}
On the one hand,
since $\eps(t)^{-(4-d)/2}\ll\sqrt{\log t}$ in the regular phase,
it follows from \eqref{Equation: log r(t) regular} that
\begin{align}
\label{Equation: Regular Localization Lower Bound 5}
\frac{\eps(t)^{-2}}{\eps(t)^{-d/2}\sqrt{\log r(t)}} \mc E(\phi)=O\left(\eps(t)^{-(4-d)/2}/\sqrt{\log t}\right)=o(1)\qquad\text{as }t\to\infty.
\end{align}
On the other hand, we note that
\begin{align}
\label{Equation: Regular Localization Lower Bound 6}
\big\langle \xi_1,(\tau_{z/\eps(t)}\phi)^2\big\rangle,\qquad z\in Z_t
\end{align}
is a Gaussian process with mean zero and covariance
\[\big\langle(\tau_{z/\eps(t)}\phi)^2,(\tau_{z'/\eps(t)}\phi)^2\big\rangle_R,\qquad z,z'\in Z_t.\]
By definition of $Z_t$ and our assumption that $\eps(t)\leq1$, if $z,z'\in Z_t$ are distinct, then the supports of
$(\tau_{z/\eps(t)}\phi)^2$ and $(\tau_{z'/\eps(t)}\phi)^2$ are separated
by at least $\ka$ in $\infty$-norm. Therefore, since we have assumed
$\ka$ to be large enough so that $R$ (and therefore $R_{\eps(t)}$ for
all $t\geq0$) is supported in $Q_{\ka/2}$, we conclude that
\eqref{Equation: Regular Localization Lower Bound 6} are i.i.d.
Gaussians with variance $\si^2\|\phi^2\|_R^2$.
In particular, given that by definition of $Z_t$ in \eqref{Equation: Regular Lower Bound Z}
and the fact that $\eps(t)$ is in the regular phase, we have
\[\log|Z_t|=\big(d\log r(t)\big)\big(1+o(1)\big),\]
then it follows from  Lemma \ref{Lemma: LLN Gaussian Maxima}
(by coupling the $(X_t(z))_{z\in Z_t}$ with a collection of i.i.d.
Gaussians with mean zero and variance $\si^2\|\phi^2\|_R^2$ on $\mbb Z$) that
\begin{align}
\label{Equation: Regular Localization Lower Bound 7}
\lim_{t\to\infty}\frac{1}{\sqrt{\log r(t)}}
\max_{z\in Z_t}\big\langle \xi_1,(\tau_{z/\eps(t)}\phi)^2\big\rangle=\si\|\phi^2\|_R\sqrt{2d}
\qquad\text{in probability}.
\end{align}

By combining the limits \eqref{Equation: Regular Localization Lower Bound 5}
and \eqref{Equation: Regular Localization Lower Bound 7}
with the lower bound \eqref{Equation: Regular Localization Lower Bound 4},
our argument so far can be summarized as follows:
For every $\ka>0$ and $\phi\in S(Q_\ka)$,
every sequence of $t>0$ such that $t\to\infty$ has a subsequence $(t_n)_{n\in\mbb N}$
along which
\[\liminf_{n\to\infty}\frac{\La_1(A^{(\si)}_{\eps(t_n)},Q_{r(t_n)})}{\eps(t_n)^{-d/2}\sqrt{\log r(t_n)}}\geq\si\|\phi^2\|_R\sqrt{2d}
\qquad\text{almost surely}.\]
Since $\phi^2\in L_1(Q_\ka)$, we can take a sequence $(\phi_n)_{n\in\mbb N}\subset S(Q_\ka)$ such that
$\phi_n^2\to\de_0$ as $n\to\infty$;
hence $\|\phi_n^2\|^2_R\to R(0)$.
Up to taking further subsequences of $t_n$, this
concludes the proof of the lower bound for \eqref{Equation: Regular Eigenvalue}.

\subsubsection{A Remark on $d\geq4$, Part 1}
\label{Section: Remark on d4 1}

In the proof of the lower bound for
\eqref{Equation: Regular Eigenvalue}
that we have just provided,
the only manifestation of the assumption that $d\leq3$
comes from the requirement that
$\eps(t)^{(d-4)/2}\ll\sqrt{\log r(t)}$, which is equivalent to
the assumption that $\eps(t)$ is in the regular phase.
If $d\geq4$, then
\begin{enumerate}
\item $\log|Z_t|=d\log\big(r(t)/\eps(t)\big)\big(1+o(1)\big)$ with $Z_t$
as in \eqref{Equation: Regular Lower Bound Z}, and
\item $\eps(t)^{(d-4)/2}\ll\sqrt{\log\big(r(t)/\eps(t)\big)}$
no matter how quickly $\eps(t)$ vanishes.
\end{enumerate}
Therefore, we get the following by using the same arguments
as in Section \ref{Section: Eigenvalues Regular Phase Lower}:

\begin{proposition}
\label{Proposition: d4 Lower Bound}
Let $d\geq4$ and $\eps(t)\in(0,1]$ be arbitrary.
Every sequence of $t>0$ such that $t\to\infty$ has a subsequence $(t_n)_{n\in\mbb N}$
along which
\[\liminf_{n\to\infty}\frac{\La_1\big(A^{(\si)}_{\eps(t_n)},w(t_n)+Q_{r(t_n)}\big)}{\eps(t_n)^{-d/2}\sqrt{\log\big(r(t_n)/\eps(t_n)\big)}}\geq
\si\sqrt{2dR(0)}\qquad\text{almost surely}.\]
\end{proposition}

\subsubsection{Step 2. Upper Bound for \eqref{Equation: Regular Eigenvalue}}
\label{Section: Eigenvalues Regular Phase Upper Bound}

Let $\eps(t)$ be in the regular phase.
Since $ \mc E$ is nonnegative and $\phi\in S(Q_{r(t)})$
are such that $\int_{\mbb R^d} \phi(x)^2\d x=1$, we have that
\[\La_1(A^{(\si)}_{\eps(t)},Q_{r(t)})\leq\si\sup_{|x|_\infty\leq r(t)/\eps(t)}\eps(t)^{-d/2}\xi_1(x).\]
We then get the desired bound by a direct application of
Lemma \ref{Lemma: LLN Gaussian Maxima}.

\subsubsection{A Remark on $d\geq4$, Part 2}
\label{Section: Remark on d4 2}

Carrying on from Section \ref{Section: Remark on d4 1},
the simple argument
in Section \ref{Section: Eigenvalues Regular Phase Upper Bound}
does not depend on the assumption that $d\geq3$. The only difference
is that if $d\geq4$ and we do not assume a lower bound on the vanishing
rate of $\eps(t)$, then $\log\big(r(t)/\eps(t)\big)$ need not be asymptotically
equivalent to $\log r(t)$. Consequently, by combining
Proposition \ref{Proposition: d4 Lower Bound} with the argument
presented in Section \ref{Section: Reduction of Eigenvalue Asymptotics},
we obtain the following result, which states that no phase transition
occurs in the eigenvalue asymptotics when $d\geq4$:

\begin{theorem}
\label{Theorem: d4}
Let $d\geq4$, and let $\eps(t)\in(0,1]$ be arbitrary.
For every $k\in\mbb N$,
\[\lim_{t\to\infty}\frac{\La_k(A_{\eps(t)},Q_t)}{\eps(t)^{-d/2}\sqrt{\log\big(t/\eps(t)\big)}}=\sqrt{2dR(0)}
\qquad\text{in probability}.\]
\end{theorem}

\subsubsection{Lower Bound for \eqref{Equation: Singular Eigenvalue}}

Let $\eps(t)$ be in the singular phase.
Let $\ka>0$ be large enough so that $R$ is supported in
$Q_{\ka/2}$. Following the heuristic in \eqref{Equation: st 3},
we expect that the leading eigenfunction of $A^{(\si)}_{\eps(t)}$
on $Q_{r(t)}$ should localize in a sub-box of size (up to a constant)
$\big(\log r(t)\big)^{-1/(4-d)}$.
Thus, in order to set up the localization lower bound in this regime, we introduce
the following lattice:
Fix some $\ka_1,\ka_2>0$, and for every $t>0$, let
\begin{align}
\label{Equation: Singular Microbox}
a(t):=\big(\ka_1\log r(t)\big)^{-1/(4-d)}
\quad\text{and}\quad
Z_t:=(2\ka_2+\ka)a(t)\mbb Z^d\cap Q_{r(t)-\ka_2a(t)}.
\end{align}

\begin{remark}
\label{Remark: Z_t lower bound size}
It is easy to see that there exists some $c>0$
such that $|Z_t|\geq cr(t)^d$ for all large enough $t$.
\end{remark}

By applying Lemma \ref{Lemma: Localization Lower Bound}
(with $\Om_i$ given by $z+Q_{\ka_2a(t)}$ for $z\in Z_t$)
and a rescaling similar to \eqref{Equation: Regular Localization Lower Bound 2}
(with $a(t)$ instead of $\eps(t)$), we have the lower bound
\begin{align}
\label{Equation: Singular Eigenvalue Lower Bound}
\La_1(A^{(\si)}_{\eps(t)},Q_{r(t)})\geq a(t)^{-2}\left(\max_{z\in Z_t}\si a(t)^2\big\langle \xi_{\eps(t)},(\tau_z\phi^{(1/a(t))})^2\big\rangle-\tfrac12 \mc E(\phi)\right)
\end{align}
for every $\phi\in S(Q_{\ka_2})$.
Until further notice, we fix a $\phi\in S(Q_{\ka_2})$.
For every $t>0$, denote the $Z_t$-indexed stochastic process
\begin{align}
\label{Equation: Ctn Lower Bound Process}
X_t(z):=\si a(t)^2\big\langle \xi_{\eps(t)},(\tau_z\phi^{(1/a(t))})^2\big\rangle,\qquad z\in Z_t.
\end{align}
By a straightforward change of variables similar to \eqref{Equation: Regular Localization Lower Bound 3},
we see that
$X_t$ is a centered stationary Gaussian process with covariance
\begin{align}
\label{Equation: Change of Var Covariance}
&\si^2a(t)^{4-2d}\int_{(\mbb R^d)^2}\phi\left(\frac{x-z}{a(t)}\right)^2R_{\eps(t)}(x-y)\phi\left(\frac{y-z'}{a(t)}\right)^2\d x\dd y\\
\nonumber
&=\si^2a(t)^{4-d}\int_{(\mbb R^d)^2}\phi\left(x-\frac{z}{a(t)}\right)^2a(t)^dR_{\eps(t)}\big(a(t)(x-y)\big)\phi\left(y-\frac{z'}{a(t)}\right)^2\d x\dd y\\
\nonumber
&=\si^2a(t)^{4-d}\big\langle\tau_{z/a(t)}\phi^2,\tau_{z'/a(t)}\phi^2\big\rangle_{R_{\eps(t)/a(t)}}.
\end{align}
We note that for any distinct $z,z'\in Z_t$ and $\phi\in S(Q_{\ka_2})$,
the supports of $\tau_{z/a(t)}\phi^2$ and $\tau_{z'/a(t)}\phi^2$
are separated by at least $\ka$ in $\ell^\infty$ norm.
Given the asymptotic \eqref{Equation: log r(t) singular},
we have that $\eps(t)/a(t)\to0$ as $t\to\infty$ when $\eps(t)$ is in the singular phase. Therefore,
at least for large $t$, $\big(X_t(z)\big)_{z\in Z_t}$ are i.i.d. random variables
with variance $\si^2a(t)^{4-d}\|\phi^2\|_{R_{\eps(t)/a(t)}}^2$
thanks to the fact that $R$ is supported in $Q_{\ka/2}$.
With this said,
we claim that a lower
bound for \eqref{Equation: Singular Eigenvalue}
is a consequence of the following:

\begin{proposition}
\label{Proposition: Continuous Lower Bound Asymptotics}
Let $X_t$ be as in \eqref{Equation: Ctn Lower Bound Process}.
For every $\ka_1,\ka_2>0$
and $\phi\in S(Q_{\ka_2})$,
\[\mbf E[X_t(0)^2]=\si^2a(t)^{4-d}\|\phi\|_4^4\big(1+o(1)\big)\qquad
\text{as }t\to\infty,\]
noting that $a(t)^{4-d}=\big(\ka_1\log r(t)\big)^{-1}$.
\end{proposition}
\begin{proof}
$\eps(t)\ll a(t)$ in the singular phase; hence
$\|\phi^2\|_{R_{\eps(t)/a(t)}}^2\to\|\phi\|_4^4$ by \eqref{Equation: Continuous Limit of Forms}.
\end{proof}

To see this, we apply a standard lower tail bound for suprema of i.i.d. Gaussians:
For large enough $t$ the $X_t(z)$ are i.i.d. copies of $X_t(0)$, and thus
\[\mbf P\left[\sup_{z\in Z_t}X_t(z)\leq\si\|\phi\|_4^2\right]=\big(1-\mbf P[X_t(0)>\si\|\phi\|_4^2]\big)^{|Z_t|}.\]
We recall the classical Gaussian tail lower bound:
If $N\sim N(0,\mf v^2)$ for some $\mf v>0$,
then for every $\th>1$ and large enough $\la>0$, one has
$\mbf P[N\geq\la]\geq\mr e^{-\th\la^2/2\mf v^2}$.
Therefore, if we assume that $0<\ka_1<2d$ and then take $t>0$ large enough so that
\[-\frac{\th\si^2\|\phi\|_4^4}{2\mbf E[X_t(0)^2]}
=-\frac{\th\big(1+o(1)\big)}{2}\ka_1\log r(t)\geq-\eta d\log r(t)\]
for some $0<\eta<1$ (by Proposition \ref{Proposition: Continuous Lower Bound Asymptotics}), then
\begin{multline*}
\big(1-\mbf P[X_t(0)>\si\|\phi\|_4^2]\big)^{|Z_t|}
\leq\left(1-\frac{1}{r(t)^{\eta d}}\right)^{|Z_t|}\\
\leq\left(1-\frac{r(t)^{d(1-\eta)}}{r(t)^d}\right)^{cr(t)^d}
\leq\mr e^{-cr(t)^{d(1-\eta)}},
\end{multline*}
where the third inequality follows from Remark \ref{Remark: Z_t lower bound size}.
By \eqref{Equation: General Radius}, $r(t)\geq\theta t^\al$ for some $\theta,\al>0$ for large $t$,
and thus every sequence of $t>0$ such that $t\to\infty$ has a subsequence $(t_n)_{n\in\mbb N}$
such that $\sum_n\mr e^{-cr(t_n)^{d(1-\eta)}}<\infty$.
Thus, by applying the Borel-Cantelli lemma to the maximum of \eqref{Equation: Ctn Lower Bound Process},
and then combining this with \eqref{Equation: Singular Eigenvalue Lower Bound}, we
conclude the following: For every $0<\ka_1<2d$, $\ka_2>0$, and $\phi\in S(Q_{\ka_2})$,
every sequence of $t>0$ such that $t\to\infty$ has a
subsequence $(t_n)_{n\in\mbb N}$ along which
\[\liminf_{n\to\infty}\frac{\La_1(A^{(\si)}_{\eps(t_n)},Q_{r(t_n)})}{a(t_n)^{-2}}\geq\si\|\phi\|_4^2-\tfrac12 \mc E(\phi)\qquad\text{almost surely}.\]
Given that $a(t)^{-2}=(\ka_1)^{2/(4-d)}(\log r(t))^{2/(4-d)}$,
up to selecting further subsequences of $t_n$, if we take $\ka_1\to2d$, $\ka_2\to\infty$,
and a sequence of $\phi$'s that achieves the supremum of $\si\|\phi\|_4^2-\tfrac12 \mc E(\phi)$
over $S(\mbb R^d)$,
then we obtain that
\[\liminf_{n\to\infty}\frac{\La_1(A^{(\si)}_{\eps(t_n)},Q_{r(t_n)})}{(\log r(t_n))^{2/(4-d)}}\geq(2d)^{2/(4-d)}\sup_{\phi\in S(\mbb R^d)}\big(\si\|\phi\|_4^2-\tfrac12 \mc E(\phi)\big)\quad\text{almost surely}.\]
This provides a lower bound for \eqref{Equation: Singular Eigenvalue}
by Proposition \ref{Proposition: L4 Sup Scaling Property} and
Lemma \ref{Lemma: Lower Bound Constant}.

\subsubsection{Upper Bound for \eqref{Equation: Singular Eigenvalue}}
\label{Section: Eigenvalues Singupar Phase Upper}

Let $\eps(t)$ be in the singular phase.
Let $\ka_1,\ka_2>0$ be fixed, and $a(t)$ be as in \eqref{Equation: Singular Microbox}.
By a straightforward rescaling (i.e., $\phi\mapsto\phi^{(1/a(t))}$ by \eqref{Equation: L2 Rescaling}
and Remark \ref{Remark: L2 Rescaling}, see also \eqref{Equation: Change of Var Covariance}), we have that
\begin{align}
\label{Equation: Singular Pre-Upper Bound}
\La_1(A^{(\si)}_{\eps(t)},Q_{r(t)})=a(t)^{-2}\La_1(A^{(\si a(t)^{(4-d)/2})}_{\eps(t)/a(t)},Q_{r(t)/a(t)}).
\end{align}
Let us define $Z_t:=2\ka_2\mbb Z_d\cap Q_{r(t)/a(t)}.$

\begin{remark}
\label{Remark: Z_t upper bound size}
By definition of $a(t)$, it is clear that there exists constants $c_1,c_2>0$
such that $|Z_t|\leq c_1r(t)^d\big(\log r(t)\big)^{c_2}$ for large $t>0$.
\end{remark}

By applying Lemma \ref{Lemma: Localization Upper Bound} to the right-hand
side of \eqref{Equation: Singular Pre-Upper Bound},
we obtain
\begin{multline}
\label{Equation: Singular Eigenvalue Upper Bound In Proof}
\La_1(A^{(\si)}_{\eps(t)},Q_{r(t)})\\
\leq
a(t)^{-2}\left(\frac{C}{\ka_2}
+\max_{z\in Z_t}\sup_{\phi\in S(Q_{\ka_2+1})}\big(\si a(t)^{(4-d)/2}\big\langle \xi_{\eps(t)/a(t)},(\tau_z\phi)^2\big\rangle-\tfrac12 \mc E(\phi)\big)\right)
\end{multline}
where the constant $C>0$ is independent of $t$, $\ka_1$, and $\ka_2$.
In order to control this quantity, we use an idea due to Chen \cite[Page 593]{Chen12}:
For every $\phi\in S(Q_{\ka_2+1})$, the function $\psi:=\phi\big(1+\tfrac12 \mc E(\phi)\big)^{-1/2}$
is an element of $W(Q_{\ka_2+1})$ because
\begin{align}
\|\psi\|_2^2+\tfrac12\mc E(\psi)=\int_{\mbb R^d}\psi(x)^2+\tfrac12|\nabla\psi(x)|_2^2\d x
=\frac{\int_{\mbb R^d}\phi(x)^2+\tfrac12|\nabla\phi(x)|_2^2\d x}{1+\tfrac12 \mc E(\phi)}=1
\end{align}
(recall the definition of the latter function space in \eqref{Equation: W Function Space}). Therefore,
\begin{align*}
&\sup_{\phi\in S(Q_{\ka_2+1})}\big(\si a(t)^{(4-d)/2}\big\langle \xi_{\eps(t)/a(t)},(\tau_z\phi)^2\big\rangle-\tfrac12 \mc E(\phi)\big)\\
&\leq\sup_{\phi\in S(Q_{\ka_2+1})}\left(\sup_{\psi\in W(Q_{\ka_2+1})}\si a(t)^{(4-d)/2}\big\langle \xi_{\eps(t)/a(t)},(\tau_z\psi)^2\big\rangle\big(1+\tfrac12 \mc E(\phi)\big)-\tfrac12 \mc E(\phi)\right)\\
&\leq\sup_{\upsilon>0}\left(\sup_{\psi\in W(Q_{\ka_2+1})}\si a(t)^{(4-d)/2}\big\langle \xi_{\eps(t)/a(t)},(\tau_z\psi)^2\big\rangle(1+\upsilon)-\upsilon\right).
\end{align*}
Given that
\begin{align}
\sup_{\upsilon>0}\big(\mf A(1+\upsilon)-\upsilon\big)
=\sup_{\upsilon>0}\big(\mf A+(\mf A-1)\upsilon\big)
=\begin{cases}
\mf A&\text{if }\mf A\leq1,\\
\infty&\text{otherwise},
\end{cases}
\end{align}
we then have the inclusion of events
\begin{multline*}
\left\{\sup_{\phi\in S(Q_{\ka_2+1})}\big(\si a(t)^{(4-d)/2}\big\langle \xi_{\eps(t)/a(t)},(\tau_z\phi)^2\big\rangle-\tfrac12 \mc E(\phi)\big)>1\right\}\\
\subset\left\{\sup_{\psi\in W(Q_{\ka_2+1})}\si a(t)^{(4-d)/2}\big\langle \xi_{\eps(t)/a(t)},(\tau_z\psi)^2\big\rangle>1\right\}
\end{multline*}
for any $z$.
Since $\xi_{\eps(t)}$ is stationary,
the $W(Q_{\ka_2+1})$-valued Gaussian processes
\[\psi\mapsto\si a(t)^{(4-d)/2}\big\langle \xi_{\eps(t)/a(t)},(\tau_z\psi)^2\big\rangle,\qquad z\in Z_t\]
are identically distributed for all $z\in Z_t$, and thus by a union bound this implies that
\begin{multline}
\label{Equation: Singular Eigenvalue Upper Bound In Proof 2}
\mbf P\left[\max_{z\in Z_t}\sup_{\phi\in S(Q_{\ka_2+1})}\big(\si a(t)^{(4-d)/2}\big\langle \xi_{\eps(t)/a(t)},(\tau_z\phi)^2\big\rangle-\tfrac12 \mc E(\phi)\big)>1\right]\\
\leq|Z_t|\cdot\mbf P\left[\sup_{\psi\in W(Q_{\ka_2+1})}\si a(t)^{(4-d)/2}\big\langle \xi_{\eps(t)/a(t)},(\tau_z\psi)^2\big\rangle>1\right].
\end{multline}
With this in hand, the upper bound for
\eqref{Equation: Singular Eigenvalue} can be reduced to
a standard Gaussian suprema
concentration bound (e.g., Lemma \ref{Lemma: Gaussian Upper Tails}) and
the following estimates:

\begin{proposition}
\label{Proposition: Continuous Upper Bound Asymptotics}
For every $\ka_1,\ka_2>0$, it holds that
\begin{align}
\label{Equation: Continuous W Asymptotics}
\sup_{\psi\in W(Q_{\ka_2+1})}\mbf E\left[\si^2a(t)^{(4-d)}\big\langle \xi_{\eps(t)/a(t)},\psi^2\big\rangle^2\right]
\leq \si^2a(t)^{4-d}\sup_{\psi\in W(\mbb R^d)}\|\psi\|_4^4<\infty
\end{align}
and
\begin{align}
\label{Equation: Continuous Dudley}
\lim_{t\to\infty}\mbf E\left[\sup_{\psi\in W(Q_{\ka_2+1})}\si a(t)^{(4-d)/2}\big\langle \xi_{\eps(t)/a(t)},\psi^2\big\rangle\right]=0.
\end{align}
\end{proposition}

To see this, let us henceforth denote for simplity,
\[\mf s:=\sup_{\psi\in W(\mbb R^d)}\|\psi\|_4^4\]
as well as the median of the supremum
\[\mf m_t:=\mbf{Med}\left[\sup_{\psi\in W(Q_{\ka_2+1})}\si a(t)^{(4-d)/2}\big\langle \xi_{\eps(t)/a(t)},(\tau_z\psi)^2\big\rangle\right]\]
and maximal standard deviation
\[\mf v_t=\sup_{\psi\in W(Q_{\ka_2+1})}\mbf E\left[\si^2a(t)^{(4-d)}\big\langle \xi_{\eps(t)/a(t)},\psi^2\big\rangle^2\right]^{1/2}.\]
Then, it follows from \eqref{Equation: Median vs Dudley} in Lemma \ref{Lemma: Gaussian Upper Tails}
that
\[\mf m_t=\mbf E\left[\sup_{\psi\in W(Q_{\ka_2+1})}\si a(t)^{(4-d)/2}\big\langle \xi_{\eps(t)/a(t)},\psi^2\big\rangle\right]
+O\left(\mf v_t\right)=o(1)\]
for any $\ka_1,\ka_2>0$, where the last equality follows from
Proposition \ref{Proposition: Continuous Upper Bound Asymptotics} and the fact that $a(t)^{4-d}=o(1)$.
In particular, we can find $t>0$ large enough so that
\[-\frac{(1-\mf m_t)^2}{2\mf v_t^2}=-\frac{1+o(1)}{2\mf v_t^2}\leq-\frac{1+o(1)}{2\si^2 a(t)^{4-d}\mf s}=-\frac{1+o(1)}{2\si^2 \mf s}\ka_1\log r(t),\]
where the third inequality follows from \eqref{Equation: Continuous W Asymptotics} and the last equality from
the definition of $a(t)$ in \eqref{Equation: Singular Microbox}.
Thus, an application of Lemma \ref{Lemma: Gaussian Upper Tails} to the present setting with $\la=1$
implies that
\[\mbf P\left[\sup_{\psi\in W(Q_{\ka_2+1})}\si a(t)^{(4-d)/2}\big\langle \xi_{\eps(t)/a(t)},(\tau_z\psi)^2\big\rangle>1\right]
\leq\exp\left(-\frac{1+o(1)}{2\si^2\mf s}\ka_1\log r(t)\right).\]
Consequently, for every choice of $\ka_1>2d\si^2\mf s$ and $\ka_2>0$, a combination of the above
bound with Remark \ref{Remark: Z_t upper bound size} and \eqref{Equation: Singular Eigenvalue Upper Bound In Proof 2}
yields
\begin{multline*}
\mbf P\left[\max_{z\in Z_t}\sup_{\phi\in S(Q_{\ka_2+1})}\big(\si a(t)^{(4-d)/2}\big\langle \xi_{\eps(t)/a(t)},(\tau_z\phi)^2\big\rangle-\tfrac12 \mc E(\phi)\big)>1\right]\\
\leq c_1r(t)^d\big(\log r(t)\big)^{c_2}\mr e^{-\de d\log r(t)}=c_1r(t)^{d(1-\de)}\big(\log r(t)\big)^{c_2}
\end{multline*}
for some $\de>1$ close enough to 1.
For any sequence of $t>0$ going to infinity,
by \eqref{Equation: General Radius} we can always extract a
sparse enough subsequence $(t_n)_{n\in\mbb N}$ such that
\[\sum_{n\in\mbb N}c_1r(t_n)^{d(1-\de)}\big(\log r(t_n)\big)^{c_2}<\infty.\]
Hence we obtain the following by an application of the Borel-Cantelli lemma:
For every choice of $\ka_1>2d\si^2\mf s$ and $\ka_2>0$,
every sequence of $t>0$ such
that $t\to\infty$ has a subsequence $(t_n)_{n\in\mbb N}$ along which
\[\limsup_{n\to\infty}\max_{z\in Z_t}\sup_{\phi\in S(Q_{\ka_2+1})}\big(\si a(t)^{(4-d)/2}\big\langle \xi_{\eps(t)/a(t)},(\tau_z\phi)^2\big\rangle-\tfrac12 \mc E(\phi)\big)\leq1\]
almost surely. By \eqref{Equation: Singular Eigenvalue Upper Bound In Proof}, this then implies that
\[\limsup_{n\to\infty}\frac{\La_1(A^{(\si)}_{\eps(t_n)},Q_{r(t_n)})}{a(t_n)^{-2}}\leq\frac{C}{\ka_2}+1
\qquad\text{almost surely},\]
which,
by taking $\ka_1\to2d\mf s$ and $\ka_2\to\infty$ and
(and further subsequences of $t_n$ if needed), yields
\[\limsup_{n\to\infty}\frac{\La_1(A^{(\si)}_{\eps(t_n)},Q_{r(t_n)})}{(\log r(t_n))^{2/(4-d)}}\leq(2d\si^2\mf s)^{2/(4-d)}
\qquad\text{almost surely}.\]
This then provides an upper bound for \eqref{Equation: Singular Eigenvalue}
by Lemma \ref{Lemma: Upper Bound Constant}.
We now conclude the proof of Theorem \ref{Theorem: Eigenvalue Technical} by proving Proposition
\ref{Proposition: Continuous Upper Bound Asymptotics}:

\begin{proof}[Proof of Proposition \ref{Proposition: Continuous Upper Bound Asymptotics}]
We begin by proving \eqref{Equation: Continuous W Asymptotics}.
By definition,
\begin{align}
\label{Equation: Continuous W Variance}
\mbf E\left[\big\langle \xi_{\eps(t)/a(t)},\psi^2\big\rangle^2\right]=\|\psi^2\|_{R_{\eps(t)/a(t)}}^2
\end{align}
for every $\psi\in W(Q_{\ka_2+1})$. Since $R_{\eps}$ integrates to one for every $\eps>0$,
it follows from Young's convolution inequality that
$\|\psi^2\|_{R_{\eps(t)/a(t)}}^2\leq\|\psi\|_4^4.$
We then get \eqref{Equation: Continuous W Asymptotics} by the trivial bound
\[\sup_{\psi\in W(Q_{\ka_2+1})}\|\psi\|_4^4\leq\sup_{\psi\in W(\mbb R^d)}\|\psi\|_4^4.\]
(The fact that the above is finite is proved in Lemma \ref{Lemma: Upper Bound Constant}.)

We now prove \eqref{Equation: Continuous Dudley}.
Let us define the pseudometrics
\[P_t(\psi,\tilde\psi):=a(t)^{(4-d)/2}
\mbf E\left[\big\langle \xi_{\eps(t)/a(t)},\psi^2-\tilde\psi^2\big\rangle^2\right]^{1/2},
\qquad \psi,\tilde\psi\in W(Q_{\ka_2+1})\]
for $t\geq0$.
Since $a(t)^{(4-d)/2}=\big(\ka_1\log r(t)\big)^{-1/2}$,
we have that
\[P_t(\psi,\tilde\psi)=\big(\ka_1\log r(t)\big)^{-1/2}\|\psi^2-\tilde\psi^2\|_{R_{\eps(t)/a(t)}}.\]
For every $\ze>0$, let us denote by $N_t(\ze)$ the covering number of $W(Q_{\ka_2+1})$
with open balls of radius $\ze$ in $P_t$. By Dudley's theorem
(e.g., \cite[Theorem 11.17]{LedouxTalagrand}), to prove \eqref{Equation: Continuous Dudley}
it is enough to show that
\begin{align}
\label{Equation: Continuous Dudley In Proof}
\lim_{t\to\infty}\int_0^\infty\sqrt{\log N_t(\ze)}\d\ze=0.
\end{align}

We first prove \eqref{Equation: Continuous Dudley In Proof}
in the case $d=1$, where we recall that we impose no lower bound on $\eps(t)$
in the singular phase.
Let us define
\[P_*(\psi,\tilde\psi)=\|\psi^2-\tilde\psi^2\|_2,\qquad\psi,\tilde\psi\in W(Q_{\ka_2+1}).\]
We note that this is the pseudometric associated to the one-dimensional Gaussian white noise.
By Young's convolution inequality, we have that
\[P_t(\psi,\tilde\psi)\leq\big(\ka_1\log r(t)\big)^{-1/2}P_*(\psi,\tilde\psi),\qquad t\geq 1,~\psi,\tilde\psi\in W(Q_{\ka_2+1}).\]
Thus, if we let $N_*(\zeta)$ denote the covering number of $W(Q_{\ka_2+1})$
by $\zeta$-balls in $P_*$, we have the inequality
\[\int_0^\infty \sqrt{\log N_t(\zeta)}\d\zeta\leq(\ka_1\log t)^{-1/2}\int_0^\infty\sqrt{\log N_*(\zeta)}\d\zeta.\]
Thanks to \cite[(2.7)]{Chen12}, we know that $\int_0^\infty \sqrt{\log N_*(\zeta)}\d\zeta<\infty$,
hence the result.

\begin{remark}
In the paper \cite{Chen12}, the one-dimensional Gaussian
white noise is referred throughout as the $``$context of Theorem 1.4$"$.
The space $W(\Om)$ for $\Om\subset\mbb R^d$ is denoted by
$\mc G_d(\Om)$ (see \cite[(2.2)]{Chen12}).
We note that the argument used to prove \cite[(2.7)]{Chen12}
cannot be extended to $d=2,3$, since a crucial assumption
in that result (i.e., \cite[(1.9)]{Chen12}) does not hold for Gaussian
white noise in $d>1$.
\end{remark}

We now prove \eqref{Equation: Continuous Dudley In Proof}
in the case $d=2,3$.

\begin{definition}
To improve readability,
for every $\psi\in W(Q_{\ka_2+1})$ and $\eps>0$,
we denote $\psi^2_\eps:=\psi^2*\bar R_\eps$
for the remainder of this proof.
\end{definition}

We begin by bounding the upper limit of integration in \eqref{Equation: Continuous Dudley In Proof}.
Recalling that, for every $\eps>0$, $R_\eps=\bar R_\eps*\bar R_\eps$
and $\bar R_\eps$ is even, we can write
\[\|\psi^2-\tilde\psi^2\|_{R_{\eps(t)/a(t)}}
=\|\psi^2_{\eps(t)/a(t)}-\tilde\psi^2_{\eps(t)/a(t)}\|_2
\leq\|\psi^2-\tilde\psi^2\|_2
\leq\|\psi\|_4^2+\|\tilde\psi\|_4^2,\]
where the first inequality follows from Young's convolution inequality.
Then, by the GNS inequality \eqref{Equation: GNS Intro} and the fact that
$\|\psi\|_2^2,\frac12 \mc E(\psi)\leq 1$ for all $\psi\in W(Q_{\ka_2+1})$, we have that
$\|\psi^2-\tilde\psi^2\|_{R_{\eps(t)/a(t)}}\leq2^{1+d/4}\mf G_d^{1/2}$.
Thus,
\begin{align}
\label{Equation: Continuous Dudley In Proof 2}
\int_0^\infty\sqrt{\log N_t(\ze)}\d\ze=\int_0^{O((\log r(t))^{-1/2})}\sqrt{\log N_t(\ze)}\d\ze.
\end{align}

Since $\eps(t)\ll a(t)$ and $\bar R$ is compactly supported, we can fix a $\ka>\ka_2$
such that $\psi^2_{\eps(t)/a(t)}\in C_0^\infty(Q_\ka)$ for every $\psi\in W(Q_{\ka_2+1})$ and $t\geq0$.
In order to estimate the covering number $N_t$, we make use of an $\eps$-net argument using the following projections:

\begin{definition}
For every $\mu,\nu,M>0$ and nonnegative $\phi\in C_0^\infty(\mbb R^d)$,
define
\[\Pi_{\mu,\nu,M}(\phi):=\sum_{z\in2\nu\mbb Z^d}\min\big\{\lfloor\phi(z)\rfloor_\mu,M\big\}\mbf 1_{z+[-\nu,\nu)^d},\]
where $\lfloor x\rfloor_\mu:=\max\{y\in\mu\mbb Z:y\leq x\}$ for every
$x\in\mbb R$.
\end{definition}

\begin{remark}
\label{Equation: Image of continuous projection}
The image of all nonnegative $\phi\in C_0^\infty(Q_{\ka})$ through
$\Pi_{\mu,\nu,M}$ has
cardinality of order $(M/\mu)^{O(\nu^{-d})}=\mr e^{O(\nu^{-d}\log(M/\mu))}$
as $\mu,\nu\to0$ and $M\to\infty$.
\end{remark}

We claim that we have the inequality
\begin{multline}
\label{Equation: Epsilon Net Approximation}
\sup_{\psi\in W(Q_{\ka_2+1})}\|\psi^2_{\eps(t)/a(t)}-\Pi_{\mu,\nu,M}(\psi^2_{\eps(t)/a(t)})\|_2\\
\leq \mf C\left(\mu+(\eps(t)/a(t))^{-(d+h)}\nu^h+M^{-1/2}\right),
\end{multline}
where $h>0$ is the H\"older exponent in \eqref{Equation: Holder Constant},
and the constant $\mf C$ only depends on $d$, $\ka$,
and the H\"older constant $C>0$ in \eqref{Equation: Holder Constant}.
In order to prove \eqref{Equation: Epsilon Net Approximation},
we use the following decomposition:
\begin{multline}
\label{Equation: Epsilon Net Approximation Two Terms}
\|\psi^2_{\eps(t)/a(t)}-\Pi_{\mu,\nu,M}(\psi^2_{\eps(t)/a(t)})\|^2_2
=\int_{\{\psi^2_{\eps(t)/a(t)}> M\}}(\psi^2_{\eps(t)/a(t)}(x)-M)^2\d x\\
+\int_{\{0<\psi^2_{\eps(t)/a(t)}\leq M\}}\Big(\psi^2_{\eps(t)/a(t)}(x)-\Pi_{\mu,\nu,M}(\psi^2_{\eps(t)/a(t)})(x)\Big)^2\d x.
\end{multline}

\begin{definition}
In what follows, we use $\mf C>0$ to denote positive constants that
(possibly) only depend on $d$, $\ka$,
and the $C$ in \eqref{Equation: Holder Constant}, and
whose exact values may change from line to line.
\end{definition}

We begin by controlling the first term on the right-hand side of \eqref{Equation: Epsilon Net Approximation Two Terms}.
Given that $(a+b)^2\leq2(a^2+b^2)$, we have
\begin{multline*}
\int_{\{\psi^2_{\eps(t)/a(t)}> M\}}(\psi^2_{\eps(t)/a(t)}(x)-M)^2\d x\\\leq
2\int_{\{\psi^2_{\eps(t)/a(t)}> M\}}\psi^2_{\eps(t)/a(t)}(x)^2\d x+2M^2\int_{\{\psi^2_{\eps(t)/a(t)}> M\}}\d x.
\end{multline*}
An application of Young's convolution inequality followed by the general $L^p$-GNS inequality
(e.g., \cite[(C.1)]{ChenBook}) implies that
\[\int_{\mbb R^d}(\psi^2_{\eps(t)/a(t)})(x)^3\d x
\leq\|\psi\|_6^6\\
\leq \mf C\|\psi\|_2^{p} \mc E(\psi)^{\tilde p}\leq\mf C,\]
where $p,\tilde p\geq0$ only depend on $d$.
Therefore, by H\"older's and Markov's inequalities,
\begin{align}\label{Equation: Equation: Epsilon Net Approximation First Term}
\nonumber
\int_{\{\psi^2_{\eps(t)/a(t)}> M\}}\psi^2_{\eps(t)/a(t)}(x)^2\d x&
\leq\left(\int_{\mbb R^d}\psi^2_{\eps(t)/a(t)}(x)^3\d x\right)^{2/3}
\left(\int_{\{\psi^2_{\eps(t)/a(t)}>M\}}\d x\right)^{1/3}\\
\nonumber
&\leq\mf C \left(M^{-3}\int_{\mbb R^d}\psi^2_{\eps(t)/a(t)}(x)^3\d x\right)^{1/3}\\
&\leq\mf CM^{-1}.
\end{align}
Applying Markov's inequality once again, we have
\begin{align}
\label{Equation: Equation: Epsilon Net Approximation First Term 2}
M^2\int_{\{\psi^2_{\eps(t)/a(t)}> M\}}\d x
\leq M^{-1}\int_{\mbb R^d}\psi^2_{\eps(t)/a(t)}(x)^3\d x\leq \mf C M^{-1}.
\end{align}

We now control the second term on the right-hand side of \eqref{Equation: Epsilon Net Approximation Two Terms}.
For any $x,y\in\mbb R^d$, since $\bar R$ is H\"older continuous
of order $h$
and $\|\psi\|_2^2\leq1$, one has
\begin{multline*}
|\psi^2_{\eps(t)/a(t)}(x)-\psi^2_{\eps(t)/a(t)}(y)|
=\left|\int_{\mbb R^d}\big(\bar R_{\eps(t)/a(t)}(x-z)-\bar R_{\eps(t)/a(t)}(y-z)\big)\psi(z)^2\d z\right|\\
\leq\mf C(\eps(t)/a(t))^{-(d+h)}|x-y|^h_2.
\end{multline*}
Thus, if $x$ is such that $|\psi^2_{\eps(t)/a(t)}(x)|\leq M$ and $x\in z+[-\nu,\nu)^d$ for some $z\in2\nu\mbb Z^d$,
\begin{align}
\Pi_{\mu,\nu,M}(\psi^2_{\eps(t)/a(t)})(x)=\lfloor\psi^2_{\eps(t)/a(t)}(z)\rfloor_\mu;
\end{align}
hence
\begin{align*}
&|\psi^2_{\eps(t)/a(t)}(x)-\Pi_{\mu,\nu,M}(\psi^2_{\eps(t)/a(t)})(x)|\\
&\leq|\psi^2_{\eps(t)/a(t)}(x)-\psi^2_{\eps(t)/a(t)}(z)|
+\big|\psi^2_{\eps(t)/a(t)}(z)-\lfloor\psi^2_{\eps(t)/a(t)}(z)\rfloor_\mu\big|\\
&\leq\mf C\left((\eps(t)/a(t))^{-(d+h)}\nu^h+\mu\right).
\end{align*}
Since $\{0<\psi^2_{\eps(t)/a(t)}\leq M\}\subset\mr{supp}(\psi^2_{\eps(t)/a(t)})\subset Q_\ka$,
we have that
\begin{multline*}
\int_{\{0<\psi^2_{\eps(t)/a(t)}\leq M\}}\big((\psi^2_{\eps(t)/a(t)})(x)-\Pi_{\mu,\nu,M}(\psi^2_{\eps(t)/a(t)})(x)\big)^2\d x\\
\leq\mf C\big((\eps(t)/a(t))^{-(d+h)}\nu^h+\mu\big)^2.
\end{multline*}
If we combine the above with \eqref{Equation: Equation: Epsilon Net Approximation First Term} and \eqref{Equation: Equation: Epsilon Net Approximation First Term 2},
we conclude that \eqref{Equation: Epsilon Net Approximation} holds.

With \eqref{Equation: Epsilon Net Approximation} established,
we are now ready to conclude the proof of \eqref{Equation: Continuous Dudley In Proof}:
Let $\mf C$ be the constant on
the right-hand side of \eqref{Equation: Epsilon Net Approximation}. Suppose that
we take
\begin{align}\label{Equation: al,be,M Conditions 1}
\mu\leq\frac{\ze}{6 \mf C},\qquad
(\eps(t)/a(t))^{-(d+h)}\nu^h\leq\frac{\ze}{6 \mf C},\qquad
M^{-1/2}\leq\frac{\ze}{6 \mf C},
\end{align}
which is equivalent to
\[\mu\leq\frac{\ze}{6\mf C},\qquad
\nu\leq\left(\frac{\ze(\eps(t)/a(t))^{d+h}}{6\mf C}\right)^{1/h},\qquad
M\geq\left(\frac{\ze}{6\mf C}\right)^{-2}.\]
Then, \eqref{Equation: Epsilon Net Approximation} implies that
\[\sup_{\psi\in W(Q_{\ka_2+1})}\|\psi^2_{\eps(t)/a(t)}-\Pi_{\mu,\nu,M}(\psi^2_{\eps(t)/a(t)})\|_2\leq\ze/2,\]
and thus any two $\psi,\tilde\psi\in W(Q_{\ka_2+1})$
such that
$\Pi_{\mu,\nu,M}(\psi^2_{\eps(t)/a(t)})
=\Pi_{\mu,\nu,M}(\tilde\psi^2_{\eps(t)/a(t)})$
will, by the triangle inequality, satisfy $P_t(\psi,\tilde\psi)\leq(\ka_1\log r(t))^{-1/2}\ze$. Therefore,
it follows from Remark \ref{Equation: Image of continuous projection} that, as $\ze\to0$,
\begin{multline*}
\sqrt{\log N_t\big(\ze/(\ka_1\log r(t))^{1/2}\big)}
=\sqrt{\log(\mr e^{O(\nu^{-d}\log(M/\mu))})}\\
=O\left(\sqrt{\zeta^{-d/h}\big(\eps(t)/a(t)\big)^{-d^2/h-d}\log(6\mf C/\zeta)}\right).
\end{multline*}
Consequently, by a change of variables, we are led to the asymptotic
\begin{multline}\label{Equation: Dudley with Epsilon Net}
\int_0^{O((\log r(t))^{-1/2})} \sqrt{\log N_t(\ze)}\d\ze\\
=O\left(\big(\log r(t)\big)^{-1/2}\big(\eps(t)/a(t)\big)^{-d^2/2h-d/2}\int_0^1\sqrt{\zeta^{-d/h}\log(6\mf C/\zeta)}\d\ze\right)
\end{multline}
as $t\to\infty$. Assuming $6\mf C\geq1$ (which we can always ensure
up to increasing the value of $\mf C$ in the upper bound \eqref{Equation: Epsilon Net Approximation})
and $h>d/4$, the integral on the right-hand side of \eqref{Equation: Dudley with Epsilon Net} is real and finite.
(Indeed, $\sqrt{\ze^{-d/h}}=\ze^{-1}$ when $h=d/4$.)
Therefore,
\[\int_0^\infty\sqrt{\log N_t(\ze)}\d\ze=O\left(\big(\log r(t)\big)^{-1/2}\big(\eps(t)/a(t)\big)^{-d^2/2h-d/2}\right).\]
Recalling that $a(t)=\big(\ka_1\log r(t)\big)^{-1/(4-d)}$,
and setting $\eps(t)\ll\big(\log r(t)\big)^{-\vartheta}$
for some $\vartheta>0$, we obtain that
\[\int_0^\infty\sqrt{\log N_t(\ze)}\d\ze=o\left(\big(\log r(t)\big)^{-(d^2+4h)/(8h-2dh)+\vartheta(d^2/2h+d/2)}\right).\]
This vanishes so long as
\[\vartheta\left(\frac{d^2}{2h}+\frac d2\right)-\frac{d^2+4h}{8h-2dh}<0\iff \vartheta<\frac{d^2+4h}{(4-d)d(d+h)}
=\frac{1}{4-d}+\frac{h}{d(d+h)}.\]
By definition of the the singular phase when $d=2,3$ (in particular \eqref{Equation: Singular Phase Lower Constant}
and the requirement $h>d/4$),
this concludes the proof of Proposition \ref{Proposition: Continuous Upper Bound Asymptotics}.
\end{proof}

\section{Quenched Total Mass Asymptotics}
\label{Section: Quenched Proof}

In this section, we prove \eqref{Equation: Main Regular TM}
and \eqref{Equation: Main Singular TM}. We begin with some
preliminary technical results in Section \ref{Section: TM Prelim},
and then prove the result in two steps in Sections
\ref{Section: Quenched Upper} and \ref{Section: Quenched Lower}.

\subsection{Preliminary Estimates}
\label{Section: TM Prelim}

\begin{proposition}
\label{Proposition: Quenched Prelim 1}
Let the function $\eta:[0,\infty)\to(0,\infty)$ be such that
\begin{align}
\label{Equation: Quenched Prelim 1}
\eta(t)=
\begin{cases}
o(1)&d=1\\
o\big(\eps(t)^{d/2}\big)&d=2,3
\end{cases},\qquad t\to\infty.
\end{align}
For every $\theta>0$, it holds in both regular and singular phases that
\[\lim_{t\to\infty} U^{(\theta)}_{\eps(t)}\big(\eta(t)\big)=1
\qquad\text{in probability,}\]
where we recall the definition of $U^{(\theta)}$ in Proposition \ref{Proposition: Quenched Upper Bound Technical}.
\end{proposition}
\begin{proof}
Since the limit is constant it suffices to show convergence in distribution.
Moreover, since constants are determined by their moments,
it suffices to prove convergence of the first two moments. For $n=1,2$,
it follows from Fubini's theorem and \eqref{Equation: Total Mass} that
\begin{align}
\label{Equation: Fubini Gaussian}
\mbf E\left[U^{(\theta)}_{\eps(t)}\big(\eta(t)\big)^n\right]
=\mbf E\left[\mbf E_{\xi_{\eps(t)}}\left[\exp\left(\theta\sum_{i=1}^n\int_0^{\eta(t)} \xi_{\eps(t)}\big(B^i(s)\big)\d s\right)\right]\right],
\end{align}
where $(B^i)_{1\leq i\leq n}$ are i.i.d. standard Brownian motions started at zero,
and $\mbf E_{\xi(t)}$ denotes expectation with respect to $\xi_{\eps(t)}$ conditional
on the $B^i$.
Conditional on a fixed realization of the paths of $B^i$, the sum of integrals
\[\theta\sum_{i=1}^n\int_0^{\eta(t)} \xi_{\eps(t)}\big(B^i(s)\big)\d s\]
is Gaussian with mean zero and variance
\begin{align}
\label{Equation: Quenched Prelim 1 Moments}
\theta^2\sum_{i,j=1}^n\int_{[0,\eta(t)]^2} R_{\eps(t)}\big(B^i(u)-B^j(v)\big)\d u\dd v,\qquad n=1,2.
\end{align}

We begin with the proof in the case $d\in\{2,3\}$.
Since $R$ is a positive semidefinite function,
$R_{\eps(t)}\leq R_{\eps(t)}(0)=\eps(t)^{-d}R(0)$.
In particular, \eqref{Equation: Quenched Prelim 1 Moments}
is bounded above by $R(0)\theta^2n^2\eta(t)^2\eps(t)^{-d}$; hence
\[1\leq\mbf E\left[U^{(\theta)}_{\eps(t)}\big(\eta(t)\big)^n\right]
=\mr e^{O((\eta(t)/\eps(t)^{d/2})^2)}=\mr e^{o(1)}=1+o(1),\]
as desired.

We now settle the case $d=1$.
For every $1\leq i\leq n$, let $(L^i_t(x))_{t\geq 1,x\in\mbb R}$
denote the continuous version of the local time process of $B^i$
(e.g., \cite[Chapter VI]{RevuzYor}), so that
\[\int_{[0,\eta(t)]^2} R_{\eps(t)}\big(B^i(u)-B^j(v)\big)\d u\dd v=\int_{\mbb R^2} L^i_{\eta(t)}(x)R_{\eps(t)}(x-y)L^j_{\eta(t)}(y)\d x\dd y.\]
Since $R_\eta$ integrates to one for all $\eta>0$, it then follows from Young's convolution inequality that
the variance in \eqref{Equation: Quenched Prelim 1 Moments} is bounded above by
\begin{align}
\label{Equation: Quenched Prelim 1 Moments 1d}
\theta^2\sum_{i,j=1}^n\|L^i_{\eta(t)}\|_2\|L^j_{\eta(t)}\|_2\leq2\theta^2\sum_{i,j=1}^n\big(\|L^i_{\eta(t)}\|_2^2+\|L^j_{\eta(t)}\|_2^2\big).
\end{align}
By Brownian scaling, $\|L^i_{\eta(t)}\|_2^2\deq\eta(t)^{3/2}\|L^i_1\|_2^2$ for all $t\geq 1$
(e.g., \cite[(2.3.8) and Proposition 2.3.5 with $d=1$ and $p=2$]{ChenBook}). Thus,
\eqref{Equation: Quenched Prelim 1 Moments 1d} converges to
zero in probability. Given that $\|L^i_1\|_2^2$ have finite
exponential moments of all orders (e.g., \cite[Theorem 4.2.1 with $p=2$]{ChenBook}),
it follows from the Vitali convergence theorem that
\[\lim_{t\to\infty}\mbf E\left[\exp\left(\eta(t)^{3/2}\theta^2\sum_{i,j=1}^n\big(\|L^i_1\|_2^2+\|L^j_1\|_2^2\big)\right)\right]=1,\]
concluding the proof.
\end{proof}

\begin{proposition}
\label{Proposition: Quenched Prelim 2}
There exists a constant $C>0$
such that for every $\theta,t,r>0$,
\[\mbf E\left[\mr e^{t\La_1(A^{(\theta)}_{\eps(t)},Q_r)}\right]\leq
\begin{cases}
\frac{2r}{\sqrt{2\pi t}}\mr e^{C\theta^4t^3}&d=1
\vspace{5pt}\\
\frac{(2r)^d}{(2\pi t)^{d/2}}\mr e^{C\theta^2\eps(t)^{-d}t^2}&d=2,3.
\end{cases}\]
in both regular and singular phases.
\end{proposition}
\begin{proof}
Thanks to \eqref{Equation: Domain Feynman-Kac} and \eqref{Equation: Trace Formula},
we note that
\begin{multline*}
\mr e^{t\La_1(A^{(\theta)}_{\eps(t)},Q_r)}
\leq\sum_{k=1}^\infty \mr e^{t\La_k(A^{(\theta)}_{\eps(t)},Q_r)}\\
=\frac1{(2\pi t)^{d/2}}\int_{Q_r}
\mbf E^{x,x}_t\left[\exp\left(\theta\int_0^t \xi_{\eps(t)}\big(B(s)\big)\d s\right);T_{Q_r}\geq t\right]\d x.
\end{multline*}
Once again employing Fubini's theorem as in \eqref{Equation: Fubini Gaussian}
and \eqref{Equation: Quenched Prelim 1 Moments},
this yields
\begin{multline*}
\mbf E\left[\mr e^{t\La_1(A^{(\theta)}_{\eps(t)},Q_r)}\right]\\
\leq\frac1{(2\pi t)^{d/2}}\int_{Q_r}
\mbf E^{x,x}_t\left[\exp\left(\frac{\theta^2}{2}\int_{[0,t]^2} R_{\eps(t)}\big(B(u)-B(v)\big)\d u\dd v\right)\right]\d x.
\end{multline*}
Given that the functional $\int_{[0,t]^2} R_{\eps(t)}\big(B(u)-B(v)\big)\d u\dd v$
is invariant with respect to the starting point of $B$, we finally get the upper bound
\begin{align}
\label{Equation: Quenched Prelim 2 - 1}
\mbf E\left[\mr e^{t\La_1(A^{(\theta)}_{\eps(t)},Q_r)}\right]
\leq\frac{(2r)^d}{(2\pi t)^{d/2}}
\mbf E^{0,0}_t\left[\exp\left(\frac{\theta^2}{2}\int_{[0,t]^2} R_{\eps(t)}\big(B(u)-B(v)\big)\d u\dd v\right)\right].
\end{align}
In the case where $d=2,3$, the result then follows from
the trivial bound
\[\int_{[0,t]^2} R_{\eps(t)}\big(B(u)-B(u)\big)\d u\dd v\leq R(0)t^2\eps(t)^{-d}.\]

We now consider the case $d=1$.
Using the same local time estimates leading up
to \eqref{Equation: Quenched Prelim 1 Moments 1d},
we have the upper bound
\[\mbf E^{0,0}_t\left[\exp\left(\frac{\theta^2}{2}\int_{[0,t]^2} R_{\eps(t)}\big(B(u)-B(v)\big)\d u\dd v\right)\right]
\leq\mbf E^{0,0}_t\left[\mr e^{\theta^2\|L_t\|_2^2/2}\right].\]
According to \cite[Lemma 2.2 in the case $d=1$ and $R=\de_0$]{ChenHSX}, for every $\th>0$,
\[\log\mbf E^0\left[\mr e^{\th\|L_t\|_2^2}\right]=O(\th^2t^3)\qquad\text{as }t\to\infty.\]
Then, by arguing as in the last paragraph of the proof of \cite[Lemma 5.11]{GaudreauLamarre}
(see also \cite[(5.15) and (5.17)--(5.19)]{GaudreauLamarre}),
we have the bound
\[\mbf E^{0,0}_t\left[\mr e^{\theta^2\|L_t\|_2^2/2}\right]=O\left(\mbf E^0\left[\mr e^{2\theta^2\|L_{t/2}\|_2^2}\right]\right),\]
thus concluding the proof for $d=1$.
\end{proof}

\subsection{Upper Bounds for
\eqref{Equation: Main Regular TM}
and \eqref{Equation: Main Singular TM}}
\label{Section: Quenched Upper}

For every $k\in\mbb N$ and $t\geq0$, define
\begin{align}
\label{Equation: rk Radius}
r_k(t):=\begin{cases}
\left(t\eps(t)^{-d/2}\sqrt{\log t}\right)^k&\text{if $\eps(t)$ is in the regular phase,}
\vspace{5pt}\\
\left(t(\log t)^{2/(4-d)}\right)^k&\text{if $\eps(t)$ is in the singular phase.}
\end{cases}
\end{align}
It is clear that, for large enough $t$, $r_k(t)< r_{k+1}(t)$ for all $k\in\mbb N$.
Consequently, following \cite[(4.24)]{GartnerKonigMolchanov}
(see also \cite[Pages 596--597]{Chen14}),
we have the decomposition
\begin{multline}
\label{Equation: Quenched Hitting Time Decomposition}
U_{\eps(t)}(t)
=\mbf E^0\left[\exp\left(\int_0^t \xi_{\eps(t)}\big(B(s)\big)\d s\right);T_{Q_{r_1(t)}}\geq t\right]\\
+\sum_{k=1}^\infty\mbf E^0\left[\exp\left(\int_0^t \xi_{\eps(t)}\big(B(s)\big)\d s\right);T_{Q_{r_k(t)}}<t\leq T_{Q_{r_{k+1}(t)}}\right].
\end{multline}

We begin by controlling the first term on the right-hand side of
\eqref{Equation: Quenched Hitting Time Decomposition}.
For the remainder of Section \ref{Section: Quenched Upper},
let us fix a some small constant $\th>0$
(precisely how small will be determined later in this proof).
By applying \eqref{Equation: Quenched Upper Bound 1}
with $r=r_1(t)$ and $\eta=t^{-\th}$, and then following this up
by \eqref{Equation: Quenched Upper Bound 2} with $r=r_1(t)$, $\tilde t=t-t^{-\th}$,
and $\theta=p$,
we obtain the upper bound
\begin{multline}
\label{Equation: Quenched Upper Bound Model}
\mbf E^0\left[\exp\left(\int_0^t \xi_{\eps(t)}\big(B(s)\big)\d s\right);T_{Q_{r_1(t)}}\geq t\right]\\
\leq U^{(q)}_{\eps(t)}(t^{-\th})^{1/q}
\big(2\pi t^{-\th}\big)^{-d/2p}\big(2 r_1(t)\big)^{d/p}\mr e^{(t-t^{-\th})\La_1(A^{(p)}_{\eps(t)},Q_{r_1(t)})/p}
\end{multline}
for every $p,q>1$ such that $1/p+1/q=1$.
Since $\eps(t)\gg(\log t)^{-1/(4-d)-\mf c_d}\gg t^{-\de}$ for all $\de>0$
when $d=2,3$,
it follows from Proposition \ref{Proposition: Quenched Prelim 1} that
\[\lim_{t\to\infty}\log U^{(q)}_{\eps(t)}(t^{-\th})^{1/q}=0\qquad\text{in probability}.\]
By definition of $r_1(t)$, we have that
\[\lim_{t\to\infty}\frac{\log \big(2\pi t^{-\th}\big)^{-d/2p}}{t}=0
\qquad\text{and}\qquad
\lim_{t\to\infty}\frac{\log r_1(t)}{t}=0.\]
Finally, noting that $t-t^{-\th}=t\big(1+o(1)\big)$, and that
$r_1(t)$ is of the form \eqref{Equation: General Radius} with $\al=1$,
it follows from Theorem \ref{Theorem: Eigenvalue Technical} and
Remark \ref{Remark: General Scaling Function} that
\[\lim_{t\to\infty}\frac{\log\mr e^{(t-t^{-\th})\La_1(A^{(p)}_{\eps(t)},Q_{r_1(t)})/p}}{t\,\eps(t)^{-d/2}\sqrt{\log t}}
=\sqrt{2dR(0)}\qquad\text{in probability}\]
in the regular phase and
\[\lim_{t\to\infty}\frac{\log\mr e^{(t-t^{-\th})\La_1(A^{(p)}_{\eps(t)},Q_{r_1(t)})/p}}
{t(\log t)^{2/(4-d)}}=p^{4/(4-d)-1}\mf L_d\quad\text{in probability}\]
in the singular phase. Combining these limits with \eqref{Equation: Quenched Upper Bound Model}
and then taking $p\to1$,
we obtain the following statement:

\begin{proposition}
\label{Proposition: Quenched Upper Bound pre}
Every sequence of $t>0$ such that $t\to\infty$ has a subsequence $(t_n)_{n\in\mbb N}$
along which the following almost-sure limits hold:
\[\limsup_{n\to\infty}\frac{\log\mbf E^0\left[\exp\left(\int_0^{t_n} \xi_{\eps(t_n)}\big(B(s)\big)\d s\right);T_{Q_{r_1(t_n)}}\geq t_n\right]}{t_n\,\eps(t_n)^{-d/2}\sqrt{\log t_n}}\leq\sqrt{2dR(0)}\]
in the regular phase, and
\[\limsup_{n\to\infty}\frac{\log\mbf E^0\left[\exp\left(\int_0^{t_n} \xi_{\eps(t_n)}\big(B(s)\big)\d s\right);T_{Q_{r_1(t_n)}}\geq t_n\right]}{t_n(\log t_n)^{2/(4-d)}}\leq\mf L_d\]
in the singular phase.
\end{proposition}

With Proposition \ref{Proposition: Quenched Upper Bound pre}
in hand, in order to complete the proof of the upper bounds for
\eqref{Equation: Main Regular TM} and \eqref{Equation: Main Singular TM},
it is enough to show that the sum on the second line of
\eqref{Equation: Quenched Hitting Time Decomposition} converges to zero in probability.
By a straightforward application of H\"older's inequality, this sum is bounded above by
\begin{multline}
\label{Equation: Quenched Sum Terms}
\sum_{k=1}^\infty\Bigg(\mbf P\left[\sup_{s\leq t}|B(s)|_{\infty}>r_k(t)\bigg|B(0)=0\right]^{1/2}\\
\cdot\mbf E^0\left[\exp\left(2\int_0^t \xi_{\eps(t)}\big(B(s)\big)\d s\right);T_{Q_{r_{k+1}(t)}}\geq t\right]^{1/2}\Bigg).
\end{multline}
Since Brownian motion suprema have sub-Gaussian tails, there exists
a $c>0$ independent of $t$ and $k$ such that, for large enough $t\geq 1$,
\[\mbf P\left[\sup_{s\leq t}|B(s)|_{\infty}>r_k(t)\bigg|B(0)=0\right]^{1/2}
\leq\mr e^{-cr_k(t)^2/t}.\]
Combining this with the upper bound used in \eqref{Equation: Quenched Upper Bound Model},
but replacing $r_1(t)$ by $r_{k+1}(t)$ and $\xi_{\eps(t)}$ by $2\xi_{\eps(t)}$, we then obtain that \eqref{Equation: Quenched Sum Terms}
is bounded above by
\begin{align}
\label{Equation: Quenched Sum Terms 2}
U^{(2q)}_{\eps(t)}(t^{-\th})^{1/2q}
\sum_{k=1}^\infty\big(2\pi t^{-\th}\big)^{-d/4p}\big(2 r_{k+1}(t)\big)^{d/2p}\mr e^{(t-t^{-\th})\La_1(
A^{(2p)}_{\eps(t)},Q_{r_{k+1}(t)})/2p-cr_k(t)^2/t}
\end{align}
for any $p,q>1$ such that $1/p+1/q=1$.
By Proposition \ref{Proposition: Quenched Prelim 1},
it suffices to prove that the sum in \eqref{Equation: Quenched Sum Terms 2}
converges to zero in probability. We analyze
the terms $k=1$ and $k\geq2$ in this sum separately.

For the term $k=1$,
we note that $r_2(t)$ is of the
form \eqref{Equation: General Radius} with exponent $\al=2$.
Thus, by Theorem \ref{Theorem: Eigenvalue Technical}, there exists a random $\bar c>0$
independent of $t$ such that
for any sparse enough diverging sequence $(t_n)_{n\in\mbb N}$,
we have that
\begin{multline*}
\frac{(t_n-t_n^{-\th})\La_1(A^{(2p)}_{\eps(t_n)},Q_{r_{2}(t_n)})}{2p}-\frac{cr_1(t_n)^2}{t_n}\\
\leq\begin{cases}
t_n\big(\bar c\eps(t_n)^{-d/2}\sqrt{\log t_n}
-c\eps(t_n)^{-d}\log t_n\big)&\text{if $\eps(t)$ is in the regular phase,}
\vspace{5pt}\\
t_n\big(\bar c(\log t_n)^{2/(4-d)}
-c(\log t_n)^{4/(4-d)}\big)&\text{if $\eps(t)$ is in the singular phase.}
\end{cases}
\end{multline*}
In particular, for every $\ka>0$,
\begin{multline*}
\big(2\pi t_n^{-\th}\big)^{-d/4p}\big(2 r_{2}(t_n)\big)^{d/2p}\mr e^{(t_n-t_n^{-\th})\La_1(
A^{(2p)}_{\eps(t_n)},Q_{r_{2}(t_n)})/2p-cr_1(t_n)^2/t_n}\\
=O\left(t_n^{\th d/4p}r_2(t_n)^{d/2p}\mr e^{-\ka t_n}\right)
\end{multline*}
almost surely as $n\to\infty$. By definition of $r_2(t)$,
this vanishes as $n\to\infty$, and thus the $k=1$ term in
the sum in \eqref{Equation: Quenched Sum Terms 2} converges
to zero in probability as $t\to\infty$.

We now deal with the terms $k\geq2$. By Proposition \ref{Proposition: Quenched Prelim 2}
(and $t-t^{-\th}\leq t$), there exists a constant $C>0$ such that
\begin{multline}
\label{Equation: Quenched Moment-Sum Asymptotic}
\mbf E\left[\sum_{k=2}^\infty\big(2\pi t^{-\th}\big)^{-d/4p}\big(2 r_{k+1}(t)\big)^{d/2p}\mr e^{(t-t^{-\th})\La_1(
A^{(2p)}_{\eps(t)},Q_{r_{k+1}(t)})/2p-cr_k(t)^2/t}\right]\\
=\begin{cases}
O\bigg(\frac{t^{\th/4p}}{t^{1/2}}\sum_{k=2}^\infty r_{k+1}(t)^{1+1/2p}\mr e^{Ct^3-cr_k(t)^2/t}\bigg)&\text{if }d=1,
\vspace{5pt}\\
O\bigg(\frac{t^{d\th/4p}}{t^{d/2}}\sum_{k=2}^\infty r_{k+1}(t)^{d+d/2p}\mr e^{C\eps(t)^{-d}t^2-cr_k(t)^2/t}\bigg)&\text{if }d=2,3.
\end{cases}
\end{multline}
We begin by controlling the right-hand side of \eqref{Equation: Quenched Moment-Sum Asymptotic}
in the case $d=1$. We note that for every $\ka_0>0$, if $t>0$ is large enough, then
$-cr_k(t)^2/t\leq-\ka_0t^{2k-1}$ for every $k\in\mbb N$. Given that $Ct^3\leq Ct^{2k-1}$
for all $t\geq1$ and $k\geq2$, for every $\ka>0$, we have that
\[\sum_{k=2}^\infty r_{k+1}(t)^{1+1/2p}\mr e^{Ct^3-cr_k(t)^2/t}
\leq\sum_{k=2}^\infty r_{k+1}(t)^{1+1/2p}\mr e^{-\ka t^{2k-1}}\]
for large enough $t$. As $r_k(t)=O(t^{\nu k})$ for some $\nu>0$
independent of $t$, this sum is uniformly bounded in $t\gg1$.
Thus, so long as we choose $\th>0$ small enough
relative to $p>1$ so that $t^{\th/4p}=o(t^{1/2})$,
we get that \eqref{Equation: Quenched Moment-Sum Asymptotic}
vanishes as $t\to0$ for $d=1$.
For $d=2,3$,
we use the same argument, noting that, since $\eps(t)\gg(\log t)^{-1/(4-d)-\mf c_d}\gg t^{-\de}$ for all $\de>0$,
$\eps(t)^{-d}t^2=O(t^3)$. In summary,
the contribution of the terms $k\geq2$ to the sum in \eqref{Equation: Quenched Sum Terms 2}
converges to zero in probability, which finally concludes the proof of the following statement:
Every sequence of $t>0$ such that $t\to\infty$ has a subsequence $(t_n)_{n\in\mbb N}$
along which
\begin{align}
\label{Equation: TM Upper Bound Final}
U_{\eps(t_n)}(t_n)\leq
\begin{cases}
\sqrt{2d R(0)}\,t_n\,\eps(t_n)^{-d/2}\sqrt{\log t_n}\big(1+o(1)\big)&\text{regular phase,}
\vspace{5pt}\\
\mf L_d\,t_n\,(\log t_n)^{2/(4-d)}\big(1+o(1)\big)&\text{singular phase.}
\end{cases}
\end{align}
almost surely as $n\to\infty$, providing upper bounds for
\eqref{Equation: Main Regular TM} and \eqref{Equation: Main Singular TM}.

\subsection{Lower Bounds for
\eqref{Equation: Main Regular TM}
and \eqref{Equation: Main Singular TM}}
\label{Section: Quenched Lower}

We now conclude the proofs of
\eqref{Equation: Main Regular TM} and \eqref{Equation: Main Singular TM}
by providing matching lower bounds to \eqref{Equation: TM Upper Bound Final}.
Let us fix some $p,q>1$ such that $1/p+1/q=1$
and some $0<\th<1$.
By \eqref{Equation: Quenched Lower Bound 1} with $\eta=r=t^\th$, we get
\begin{multline}
\label{Equation: Quenched Lower Bound Decomposition}
U_{\eps(t)}(t)\geq U^{(-q/p)}_{\eps(t)}(t^\th)^{-p/q}\\
\cdot\left(\int_{Q_{t^\th}}\ms G_{t^\th}(x)\,
\mbf E^x\left[\exp\left(\frac 1p\int_0^{t-t^\th} \xi_{\eps(t)}\big(B(s)\big)\d s\right);T_{Q_{t^\th}}\geq t-t^\th\right]\d x\right)^p.
\end{multline}
Firstly, we note that $-\xi_{\eps(t)}$ is equal in distribution to $\xi_{\eps(t)}$,
and thus the asymptotics of $U^{(-q/p)}_{\eps(t)}(t^\th)^{-p/q}$
are the same as that of $U^{(q/p)}_{\eps(t)}(t^\th)^{-p/q}$.
Secondly, it is easy to see that $\eps(t)$ is in the regular (resp. singular)
phase if and only of $\eps(t^{1/\th})$ is in the regular (resp. singular) phase
for every $\th>0$. Consequently, it follows from \eqref{Equation: TM Upper Bound Final}
that if the diverging sequence $(t_n)_{n\in\mbb N}$ is sufficiently sparse, then
\[\frac{\log U^{(-q/p)}_{\eps(t_n)}(t_n^\th)^{-p/q}}{t_n}=
\begin{cases}
O\left(t_n^{\th-1}\eps(t_n)^{-d/2}\sqrt{\log t_n}\right)&\text{in the regular phase}
\vspace{5pt}\\
O\left(t_n^{\th-1}(\log t_n)^{2/(4-d)}\right)&\text{in the singular phase}
\end{cases}\]
almost surely as $n\to\infty$. Since $\th<1$ this vanishes for large $n$.

We now analyze the term on the second line of \eqref{Equation: Quenched Lower Bound Decomposition}.
It is easy to see that there exists a constant $c>0$
such that $\ms G_{t^\th}(x)\geq c\mr e^{-ct^\th}$
for every $x\in(-t^\th,t^\th)^d$. Thus,
an application of \eqref{Equation: Quenched Lower Bound 2}
with $\tilde t=t-t^\th$, $\eta=r=t^\th$, and $\theta=1/p$ yields that
the term on the second line of \eqref{Equation: Quenched Lower Bound Decomposition}
is bounded below by the quantity
\begin{multline*}
\mc F_{p,q,\th}(t):=c\mr e^{-cpt^\th}(2\pi)^{p^2d/2}t^{dp\th /2}(t-t^\th)^{p^2d/2q}(2t^\th)^{-2p^2/q}\\
\cdot\mr e^{-t^\th(p^2/q)\La_1(A^{(q/p^2)}_{\eps(t)},Q_{t^\th})}\mr e^{p^2t\La_1(A^{(1/p^2)}_{\eps(t)},Q_{t^\th})}
\end{multline*}
Since $t^\th$ is of the form \eqref{Equation: General Radius}
with exponent $\al=\th$, we conclude from Theorem \ref{Theorem: Eigenvalue Technical}
and Remark \ref{Remark: General Scaling Function} that
\[\lim_{t\to\infty}\frac{\log \mc F_{p,q,\th}(t)}{t\,\eps(t)^{-d/2}\sqrt{\log t}}=\sqrt{2dR(0)\th}
\qquad\text{in probability}\]
in the regular phase and
\[\lim_{t\to\infty}\frac{\log \mc F_{p,q,\th}(t)}{t\,(\log t)^{2/(4-d)}}
=p^{2-8/(4-d)}\th^{2/(4-d)}\mf L_d
\qquad\text{in probability}\]
in the singular phase. By taking $p,\th\to1$, this yields
a lower bound for
\eqref{Equation: Main Regular TM} and \eqref{Equation: Main Singular TM},
thus concluding the proof of Theorems
\ref{Theorem: Main Regular} and \ref{Theorem: Main Singular}.

\section{Annealed Total Mass}
\label{Section: Annealed Proof}

We now prove Theorem \ref{Theorem: Annealed}. Our main tool in establishing
this is the following moment formula,
which is proved using the same Fubini computation as in \eqref{Equation: Fubini Gaussian}:
For every $p\in\mbb N$,
\begin{align}
\label{Equation: Annealed Expression}
\mbf E\big[U_{\eps(t)}(t)^p\big]=\mbf E\left[\exp\left(\frac12\sum_{i,j=1}^p\int_{[0,t]^2}
R_{\eps(t)}\big(B^i(u)-B^j(v)\big)\d u\dd v\right)\right],
\end{align}
where $(B_i)_{1\leq i\leq p}$ are i.i.d. standard Brownian motions on $\mbb R^d$ started at zero.

\subsection{Proof of \eqref{Equation: Regular Annealed Total Mass}}

Suppose that $d\geq2$, or $d=1$ and $\eps(t)^{-1}=o(t)$.
Informally speaking, the statement of \eqref{Equation: Regular Annealed Total Mass}
is that, under these assumptions, the only meaningful contribution of
\eqref{Equation: Annealed Expression} in the large $t$ limit comes from
paths of the Brownian motions that are confined to a neighborhood of the origin,
whereby
\[\int_{[0,t]^2}R_{\eps(t)}\big(B^i(u)-B^j(v)\big)\approx t^2R_{\eps(t)}(0)=\eps(t)^{-d}t^2R(0).\]
Given that $R\leq R(0)$ by virtue of being a covariance function, an upper bound to that effect is trivial:
\[\int_{[0,t]^2}R_{\eps(t)}\big(B^i(u)-B^j(v)\big)\d u\dd v\leq \eps(t)^{-d}t^2R(0)\]
for every $1\leq i,j\leq p$,
which immediately yields an upper bound for \eqref{Equation: Regular Annealed Total Mass}.
To prove a matching lower bound, we now argue that confining the paths of the Brownian
motions near zero when $d\geq2$, or $d=1$ and $\eps(t)^{-1}=o(t)$ yields a vanishing
error.

Let $\ka>0$ be fixed. If $|B^i(s)|_\infty\leq\ka\eps(t)$ for every $1\leq i\leq p$ and $0\leq s\leq t$, then
\[\sum_{i,j=1}^p\int_{[0,t]^2}R_{\eps(t)}\big(B^i(u)-B^j(v)\big)\d u\dd v\geq\eps(t)^{-d}t^2p^2\inf_{|x|_\infty,|y|_\infty\leq\ka}R(x-y).\]
Given that $\eps(t)$ is bounded above by $1$, $\eps(t)/\sqrt{t}\to0$ as $t\to\infty$,
and thus there exists a constant $C>0$ independent of $t$ such that for each $1\leq i\leq p$,
\begin{align}
\label{Equation: Brownian Suprema Ball}
\mbf P\left[\sup_{0\leq s\leq t}|B^i(s)|_\infty\leq\ka\eps(t)\right]\geq\mr e^{-C t/\ka^2\eps(t)^2}
\end{align}
for large $t>0$
(e.g., \cite[(1.3), Page 535]{LiChao}).
Thanks to \eqref{Equation: Annealed Expression}, we have the inequality
\begin{multline}
\label{Equation: Events Lower Bound}
\mbf E\big[U_{\eps(t)}(t)^p\big]
\geq\mbf E\Bigg[\exp\left(\frac{1}2\sum_{i,j=1}^p\int_{[0,t]^2}R_{\eps(t)}\big(B^i(u)-B^j(v)\big)\d u\dd v\right)\\
\cdot\prod_{i=1}^p\mbf 1_{\{\sup_{0\leq s\leq t}|B^i(s)|_\infty\leq\ka\eps(t)\}}\Bigg];
\end{multline}
consequently,
\begin{align}
\label{Equation: Subcritical Annealed Lower Bound Precursor}
\liminf_{t\to\infty}\frac{\log\mbf E\big[U_{\eps(t)}(t)^p\big]}{\eps(t)^{-d}t^2}\geq\frac{p^2}2\inf_{|x|_\infty,|y|_\infty\leq\ka}R(x-y)
+\liminf_{t\to\infty}
\frac{-C p\eps(t)^{d-2}}{\ka^2t}.
\end{align}
The liminf on the right-hand side of \eqref{Equation: Subcritical Annealed Lower Bound Precursor} vanishes
whenever $d\geq2$ (since $\eps(t)\leq1$), or $d=1$ and $\eps(t)^{-1}=o(t)$.
Given that \eqref{Equation: Subcritical Annealed Lower Bound Precursor} holds for arbitrarily small $\ka>0$,
we thus obtain a lower bound for the limit \eqref{Equation: Regular Annealed Total Mass}
by taking $\ka\to0$.

\subsection{Proof of \eqref{Equation: Singular Annealed Total Mass}}

Let us henceforth assume that $d=1$ with $\eps(t)\ll t^{-1}$.
As per \eqref{Equation: Subcritical Annealed Lower Bound Precursor},
once $\eps(t)$ passes the threshold of $t^{-1}$ in one dimension,
the space scaling of $R_{\eps(t)}$ vanishes too quickly
for only Brownian paths confined to zero to have a contribution.
In the case of the upper bound, we can simply take the $\eps\to0$
limit in \eqref{Equation: Annealed Expression}, whereby the asymptotics
of the exponential moment reduce to the large deviations of Brownian
motion local time. More specifically:
Arguing as in
\eqref{Equation: Quenched Prelim 1 Moments 1d}, we see that
\[\sum_{1\leq i,j\leq p}\int_{[0,t]^2}R_{\eps(t)}\big(B^i(u)-B^j(v)\big)\d u\dd v
\leq2\sum_{i,j=1}^p\big(\|L^i_t\|_2^2+\|L^j_t\|_2^2\big)=4p\sum_{i=1}^p\|L_t^i\|_2^2.\]
By independence of the Brownian motions $B^i$ and \eqref{Equation: Annealed Expression}, this means that
\[\mbf E\big[U_{\eps(t)}(t)^p\big]\leq\mbf E^0\left[\mr e^{4p\|L_t\|_2^2}\right]^p,\]
where $L_t$ denotes the local time of some Brownian motion $B$, under the expectation $\mbf E^0$.
According to \cite[Lemma 2.2 in the case $d=1$ and $R=\de_0$]{ChenHSX}, there exists
a constant $C>0$ such that for every $\th>0$,
\[\lim_{t\to\infty}\frac{\log\mbf E^0\left[\mr e^{\th\|L_t\|_2^2}\right]^p}{t^3}=Cp\th^2,\]
from which we immediately obtain the upper bound in \eqref{Equation: Singular Annealed Total Mass}.

It now remains to show that there is a lower bound of the same order (i.e.,
$p^3t^3$).
The argument that we use for this is inspired by the proof of \cite[(6.8)]{Tindeletal}.
That is, we introduce an additional smoothing of the noise, which allows to simultaneously
provide a lower bound for the moment \eqref{Equation: Annealed Expression} and capture the optimal
range of the Brownian paths that contribute to the Annealed asymptotics; once this is done
the argument follows as in \eqref{Equation: Subcritical Annealed Lower Bound Precursor}. For this purpose,
we begin with the following Lemma, in which we introduce the additional smoothing:

\begin{lemma}
For every $\eta>0$, it holds that
\begin{multline}
\label{Equation: ???}
\sum_{i,j=1}^p\int_{[0,t]^2}R_{\eps(t)}\big(B^i(u)-B^j(v)\big)\d u\dd v\\\geq
\sum_{i,j=1}^p\int_{[0,t]^2}(R_{\eps(t)}*\ms G_{\eta})\big(B^i(u)-B^j(v)\big)\d u\dd v,
\end{multline}
where we recall that $\ms G_t$ denotes the Gaussian kernel defined
in \eqref{Equation: Gaussian Kernel}.
\end{lemma}
\begin{proof}
Recall that we can write
\[\int_{[0,t]^2}f\big(B^i(u)-B^j(v)\big)\d u\dd v=\int_{\mbb R^2}L^i_t(x)f(x-y)L^j_t(y)\d x\dd y\]
for any measurable $f:\mbb R\to\mbb R$, where $L^i$ denotes the local time process
of $B^i$. Thus, if we denote $\mbf L_t(x):=\sum_{i=1}^pL^i_t(x)$, then we have that
\begin{align}
\label{Equation: Sum of Local Times}
\sum_{1\leq i,j\leq p}\int_{[0,t]^2}f\big(B^i(u)-B^j(v)\big)\d u\dd v=\int_{\mbb R^2}\mbf L_t(x)f(x-y)\mbf L_t(y)\d x\dd y.
\end{align}
Letting $\hat \cdot$ denote the Fourier transform,
it follows from the Parseval formula that
\begin{align}
\label{Equation: Fourier Equation}
\int_{\mbb R^2}\mbf L_t(x)f(x-y)\mbf L_t(y)\d x\dd y
=\int_{\mbb R}|\hat{\mbf L_t}(x)|^2\hat f(x)\d x.
\end{align}
For every $\eta>0$, $\bar R_\eta$ and $\ms G_\eta$ are both even functions,
and $R_\eta=\bar R_\eta*\bar R_\eta$ and $\ms G_\eta=\ms G_{\eta/2}*\ms G_{\eta/2}$.
In particular, $\hat{R_\eta}$ and $\hat{\ms G_\eta}$ are both nonnegative.
Given that
\[\hat{\ms G_\eta}\leq\|\ms G_\eta\|_1=1\]
the result then follows from the applying the inequality
\[\hat{R_{\eps(t)}*\ms G_\eta}=\hat{R_{\eps(t)}}\cdot\hat{\ms G_\eta}\leq\hat{R_{\eps(t)}}.\]
to \eqref{Equation: Fourier Equation} with
$f=R_{\eps(t)}*\ms G_\eta$.
\end{proof}

Let $\ka>0$ be large enough so that $\mr{supp}(R_{\eps(t)})\subset[-\ka\eps(t),\ka\eps(t)]$
for every $t\geq 0$; in particular, for every $\tilde\ka,\eta>0$, we have that
\[\inf_{|x|\leq\tilde\ka}(R_{\eps(t)}*\ms G_\eta)(x)\geq\inf_{|x|\leq\tilde\ka+\ka\eps(t)}\ms G_\eta(x)=\frac{\mr e^{-(\tilde\ka+\ka\eps(t))^2/2\eta}}{\sqrt{2\pi\eta}}.\]
Let us define a function $\eta(t)$ that vanishes as $t\to0$ in such a way that $\eps(t)=o(\eta(t)^{1/2})$
(we define $\eta(t)$ more specifically in a moment; this function is meant to capture
the optimal range of Brownian paths that contribute to $\mbf E[U(t)^p]$).
If $|B^i(s)|\leq\eta(t)^{1/2}$ for every $0\leq s\leq t$ and $0\leq i\leq p$,
then we have the inequality
\begin{align}
\label{Equation: Gaussian Convolution Lower Bound}
\sum_{i,j=1}^p\int_{[0,t]^2}(R_{\eps(t)}*\ms G_{\eta(t)})\big(B^i(u)-B^j(v)\big)\d u\dd v
\geq p^2t^2\frac{\mr e^{-(2\eta(t)^{1/2}+\ka\eps(t))^2/2\eta(t)}}{\sqrt{2\pi\eta(t)}}.
\end{align}
Since $\eps(t)^2=o(\eta(t))$, there exists some $c>0$ such that
\eqref{Equation: Gaussian Convolution Lower Bound} is bounded below by
$cp^2t/\eta(t)^{1/2}$
for large enough $t$. Using essentially the same estimates as \eqref{Equation: Brownian Suprema Ball}
and \eqref{Equation: Events Lower Bound}, we therefore conclude that
\[\liminf_{t\to\infty}\frac{\log\mbf E\big[U_{\eps(t)}(t)^p\big]}{t^3}\geq
\liminf_{t\to\infty}\frac{1}{t^3}\left(\th_1\frac{p^2t^2}{\eta(t)^{1/2}}-\th_2\frac{pt}{\eta(t)}\right)\]
for some $\th_1,\th_2>0$ independent of $p$ and $t\geq 1$. If we take
\[\eta(t)=\frac{\bar\ka\th_2^2}{\th_1^2p^2t^2}\]
for some fixed $\bar\ka>1$ (which statisfies $\eps(t)=o(\eta(t)^{1/2})$ thanks to our assumption that $\eps(t)=o(t^{-1})$),
we get that
\[\frac{1}{t^3}\left(\th_1\frac{p^2t^2}{\eta(t)^{1/2}}-\th_2\frac{pt}{\eta(t)}\right)=\frac{\th_1^2 \left(\sqrt{\bar\ka}-1\right)}{\th_2 \bar\ka}p^3,\]
which yields the lower bound in \eqref{Equation: Singular Annealed Total Mass}.

\appendix
\section{Appendix}

\subsection{Gaussian Maxima Upper Tails}

\begin{lemma}[e.g.,
{\cite[Theorem 5.4.3 and Corollary 5.4.5]{MarcusRosen}}]
\label{Lemma: Gaussian Upper Tails}
Let $\big(X(x)\big)_{x\in \mf I}$ be a centered Gaussian process
such that $\mf I$ is a countable metric space.
Denote the maximal variance and median of $X$ as
\[\mf v:=\sup_{x\in\mf I}\mbf{E}\big[X(x)^2\big]^{1/2}
\qquad\text{and}\qquad
\mf m:=\mbf{Med}\left[\sup_{x\in\mf I}X(x)\right].\]
For every $\la\geq 0$
\begin{align}
\label{Equation: Marcus Tail Bound}
\mbf P\left[\sup_{x\in\mf I}X(x)>\la\right]
\leq\mr e^{-(\la-\mf m)^2/2\mf v^2}.
\end{align}
Moreover,
\begin{align}\label{Equation: Median vs Dudley}
\left|\mf m-\mbf E\left[\sup_{x\in\mf I}X(x)\right]\right|\leq\frac{\mf v}{\sqrt{2\pi}}.
\end{align}
\end{lemma}

\subsection{Variations and Best Constants}

The proofs of the results in this section are standard
in the large deviation literature (e.g., \cite{ChenBook,Chen12}).
We nevertheless provide the arguments in full for the reader's convenience.

\subsubsection{Scaling Property}

\begin{proposition}
\label{Proposition: L4 Sup Scaling Property}
Let $c>0$ be fixed.
For every $\eta>0$, it holds that
\[\sup_{\phi\in S(\mbb R^d)}\big(c\|\phi\|_4^2-\tfrac12 \mc E(\phi)\big)
=\eta^2\sup_{\phi\in S(\mbb R^d)}\big(\eta^{(d-4)/2}c\|\phi\|_4^2-\tfrac12 \mc E(\phi)\big).\]
\end{proposition}
\begin{proof}
This follows from a direct application of Remark \ref{Remark: L2 Rescaling},
as $\phi\in S(\mbb R^d)$ if and only if $\phi^{(\eta)}\in S(\mbb R^d)$.
\end{proof}

\subsubsection{Equivalence}

\begin{lemma}\label{Lemma: Lower Bound Constant}
Recall the definition of $\mf G_d$ as the smallest constant in the inequality
\[\|\phi\|_4^4\leq \mf G_d\mc E(\phi)^{d/2}\|\phi\|_2^{4-d}
\qquad\text{for all }\phi\in C_0^\infty(\mbb R^d).\]
It holds that
\[\sup_{\phi\in S(\mbb R^d)}\big(\|\phi\|_4^2-\tfrac12 \mc E(\phi)\big)
=\frac{4-d}4 \left(\frac d2\right)^{d/(4-d)}\mf G_d^{2/(4-d)}.\]
\end{lemma}
\begin{proof}
We begin with an upper bound.
By definition of $\mf G_d$, for every $\phi\in S(\mbb R^d)$,
\[\|\phi\|_4^2-\tfrac12 \mc E(\phi)
\leq\mf G_d^{1/2} \mc E(\phi)^{d/4}-\tfrac12 \mc E(\phi)
\leq\sup_{x\geq0}\big(\mf G_d^{1/2}x^{d/4}-\tfrac12x\big)
=\frac{4-d}4 \left(\frac d2\right)^{d/(4-d)}\mf G_d^{2/(4-d)},\]
where the last equality follows from elementary calculus.
For a matching lower bound,
let $0<C<\mf G_d$. Then, there exists $\phi\in S(\mbb R^d)$ such that $\|\phi\|_4^4\geq C \mc E(\phi)^{d/2}$.
By Remark \ref{Remark: L2 Rescaling}, we see that
\[\|(\phi^{(\eta)})^2\|_2-\tfrac12 \mc E(\phi^{(\eta)})>C^{1/2}\big(\eta^2 \mc E(\phi)\big)^{d/4}-\tfrac12\eta^2 \mc E(\phi).\]
Since $\eta>0$ was arbitrary, we conclude that
\[\sup_{\phi\in S(\mbb R^d)}\big(\|\phi\|_4^2-\tfrac12 \mc E(\phi)\big)\geq\sup_{x\geq0}\big(C^{1/2}x^{d/4}-\tfrac12x\big),\]
which yields the desired lower bound by taking $C\to\mf G_d$.
\end{proof}

\begin{lemma}\label{Lemma: Upper Bound Constant}
With $\mf G_d$ defined as in Lemma \ref{Lemma: Lower Bound Constant}, we have that
\begin{align}
\label{Equation: Upper Bound Constant}
\sup_{\psi\in W(\mbb R^d)}\|\psi\|_4^4=\left(\frac{4-d}4\right)^{(4-d)/2}\left(\frac d2\right)^{d/2}\mf G_d.
\end{align}
In particular, recalling the definition of $\mf L_d$ in \eqref{Equation: Lyapunov}, we have that
\[\left(2d\sup_{\psi\in W(\mbb R^d)}\|\psi\|_4^4\right)^{2/(4-d)}=\mf L_d.\]
\end{lemma}
\begin{proof}
For simplicity of notation, let us denote
\[\mf s:=\sup_{\psi\in W(\mbb R^d)}\|\psi\|_4^4.\]
We first prove that $\mf s$ is finite.
$\phi\in S(\mbb R^d)$ if and only if
$\phi\big(1+\tfrac12 \mc E(\phi)\big)^{-1/2}\in W(\mbb R^d)$.
Therefore, an application of \eqref{Equation: GNS Intro} yields
\[\mf s=\sup_{\phi\in S(\mbb R^d)}\frac{\|\phi\|_4^4}{(1+\frac12\mc E(\phi))^2}\leq\mf G_d\sup_{x\geq0}\frac{x^{d/2}}{(1+\frac x2)^2},\]
which is finite for $d=1,2,3$.
We now prove \eqref{Equation: Upper Bound Constant}: Note that
\[\|\phi\|_4^2-\tfrac12\mf s^{1/2} \mc E(\phi)\leq
\mf s^{1/2}\big(1+\tfrac12 \mc E(\phi)\big)-\tfrac12\mf s^{1/2} \mc E(\phi)=\mf s^{1/2}.\]
Thus, by applying Proposition \ref{Proposition: L4 Sup Scaling Property}
with $c=\mf s^{-1/2}$ and $\eta=\mf s^{1/(d-4)}$,
and then Lemma \ref{Lemma: Lower Bound Constant},
we obtain that
\begin{multline*}
1\geq\mf s^{-1/2}\sup_{\phi\in S(\mbb R^d)}\big(\|\phi\|_2-\tfrac12\mf s^{1/2} \mc E(\phi)\big)
=\sup_{\phi\in S(\mbb R^d)}\big(\mf s^{-1/2}\|\phi\|_2-\tfrac12 \mc E(\phi)\big)\\
=\mf s^{-2/(4-d)}\cdot\frac{4-d}4 \left(\frac d2\right)^{d/(4-d)}\mf G_d^{2/(4-d)}.
\end{multline*}
Solving for $\mf s$ in the above inequality yields
\[\mf s\geq\left(\frac{4-d}4\right)^{(4-d)/2}\left(\frac d2\right)^{d/2}\mf G_d.\]
We now provide a matching upper bound. For every $\phi\in C_0^\infty(\mbb R^d)$,
\[\|\phi\|_4^2\leq\mf G_d^{1/2} \mc E(\phi)^{d/4}\|\phi\|_2^{(4-d)/2}
=\mf G_d^{1/2}\left(\frac{4-d}{2d}\right)^{-d/4}\left(\frac{4-d}{2d} \mc E(\phi)\right)^{d/4}\big(\|\phi\|_2^2\big)^{(4-d)/4}.\]
Next, we use Young's classical inequality $|xy|\leq |x|^p/p+|y|^q/q$ for $1/p+1/q=1$ in the
special case $p=4/(4-d)$ and $q=4/d$, which yields
\[\|\phi\|_4^2\leq
\mf G_d^{1/2}\left(\frac{4-d}{2d}\right)^{-d/4}\left(\frac{4-d}4\right)\big(\|\phi\|_2^2+\tfrac12 \mc E(\phi)\big).\]
If we divide both sides by $\|\phi\|_2^2+\tfrac12 \mc E(\phi)$ and take a supremum over
smooth and compactly supported $\phi$, then we get that
\[\mf s\leq\mf G_d\left(\frac{4-d}{2d}\right)^{-d/2}\left(\frac{4-d}4\right)^2
=\left(\frac{4-d}4\right)^{(4-d)/2}\left(\frac d2\right)^{d/2}\mf G_d,\]
concluding the proof.
\end{proof}

\begin{acknowledgements}
The author thanks Mykhaylo Shkolnikov for numerous discussions
on the content and presentation of the paper.
The author thanks Martin Hairer, Cyril Labb\'e, and
Willem van Zuijlen for insightful questions and comments,
and the latter for sharing the article \cite{KPZ}.

The author thanks anonymous referees for a careful reading
of a previous version of this paper, including the identification of
several errors.
\end{acknowledgements}

%
%

\bibliographystyle{spmpsci}      
\bibliography{Bibliography}      


\end{document}